\documentclass[10pt,reqno]{amsart}

\usepackage{a4wide}
\usepackage{dsfont}            % used in the commands \I, \Ww, \wW, \WW
\usepackage{amssymb,amsmath,amsfonts}
\usepackage{mathrsfs}
\usepackage{graphicx} 
\usepackage{xcolor}          % used in the commands \wW, \WW
\usepackage[latin10]{inputenc} % needed for Stratila's initial
\usepackage[all]{xy}

\newtheorem{proposition}{Proposition}[section]
  \newtheorem{theorem}[proposition]{Theorem}
  \newtheorem{corollary}[proposition]{Corollary}
  \newtheorem{lemma}[proposition]{Lemma}
\theoremstyle{definition}
  \newtheorem{definition}[proposition]{Definition}
  \newtheorem{remark}[proposition]{Remark}
   \newtheorem{notation}[proposition]{Notation}
  \newtheorem{example}[proposition]{Example}
\allowdisplaybreaks
\newcommand{\cst}{\ifmmode\mathrm{C}^*\else{$\mathrm{C}^*$}\fi}
\newcommand{\st}{\;\vline\;}

\newcommand{\CC}{\mathbb{C}}
\newcommand{\RR}{\mathbb{R}}

\newcommand{\GG}{\mathbb{G}}

\newcommand{\vtens}{\,\bar{\otimes}\,}
\newcommand{\id}{\mathrm{id}}

\newcommand{\I}{\mathds{1}}

\newcommand{\HH}{\mathbb{H}}
\newcommand{\sM}{\mathsf{M}}
\newcommand{\sN}{\mathsf{N}}
\newcommand{\sB}{\mathsf{B}}
\newcommand{\sD}{\mathsf{D}}
\newcommand{\sA}{\mathsf{A}}
\newcommand{\sX}{\mathsf{X}}

\newcommand{\hh}[1]{\widehat{#1}}

\newcommand{\flip}{\boldsymbol{\sigma}}
\newcommand{\ww}{\mathrm{W}}
\newcommand{\WW}{{\mathds{V}\!\!\text{\reflectbox{$\mathds{V}$}}}}
\newcommand{\Ww}{\mathds{W}}
\newcommand{\wW}{\text{\reflectbox{$\Ww$}}\:\!} % requires graphicx, only works in pdf via pdflatex
\newcommand{\Vv}{\mathds{V}}
\newcommand{\vV}{\text{\reflectbox{$\Vv$}}\:\!}
\newcommand{\U}{\mathrm{U}}
\newcommand{\Uu}{\mathds{U}}
\newcommand{\dd}[1]{\widetilde{#1}}

\newcommand{\starr}{\,\underline{*}\,}
\newcommand{\staru}{\,\overline{*}\,}

\DeclareMathOperator{\C}{C}
\DeclareMathOperator{\B}{B}

\DeclareMathOperator{\Mor}{Mor}
\DeclareMathOperator{\M}{M}
\DeclareMathOperator{\N}{N}

\DeclareMathOperator{\Linf}{\mathnormal{L}^\infty\;\!\!}
\DeclareMathOperator{\Ltwo}{\mathnormal{L}^2\;\!\!}

\numberwithin{equation}{section}

\def\labelitemi{$\blacktriangleright$}

\author{Pawe{\l} Kasprzak}
\address{Department of Mathematical Methods in Physics, Faculty of Physics, University of Warsaw, Poland}
\email{pawel.kasprzak@fuw.edu.pl}

\author{Fatemeh Khosravi}
\address{Department of Pure Mathematics, Ferdowsi University of Mashhad, Iran}
\email{fa.khosravi@stu-mail.um.ac.ir}

\title[Coideals,  quantum subgroups and idempotent states]{Coideals,  quantum subgroups and idempotent states}



\begin{document}

\begin{abstract} We establish a one to one correspondence between idempotent states on a locally compact quantum group $\GG$ and integrable coideals in the von Neumann algebra $\Linf(\GG)$ that are preserved by the scaling group. In particular we show that there is a  one to one correspondence between  idempotent states on $\GG$ and $\psi_\GG$-expected left-invariant von Neumann subalgebras of $\Linf(\GG)$. 
We characterize  idempotent states of Haar type as those corresponding to  integrable normal coideals preserved by the scaling group. We also establish a one to one correspondence between  open subgroups of $\GG$ and   central   idempotent states on the dual $\hh\GG$. Finally we characterize coideals corresponding to  open quantum subgroups of $\GG$ as those that are normal and  admit an atom. 
As a byproduct of this study we get a number of universal lifting results for Podle\'s condition, normality and regularity and we generalize a number of results known before to hold  under the coamenability assumption. 
\end{abstract}

\maketitle

% \tableofcontents

\newlength{\sw}
\settowidth{\sw}{$\scriptstyle\sigma-\text{\rm{weak closure}}$}
\newlength{\nc}
\settowidth{\nc}{$\scriptstyle\text{\rm{norm closure}}$}
\newlength{\ssw}
\settowidth{\ssw}{$\scriptscriptstyle\sigma-\text{\rm{weak closure}}$}
\newlength{\snc}
\settowidth{\snc}{$\scriptscriptstyle\text{\rm{norm closure}}$}
\renewcommand{\labelitemi}{$\bullet$}

\section{Introduction}\label{intro}
Locally compact quantum groups theory  is formulated in terms of operator algebras. The system of axioms becomes particularly simple when written in the language of von Neumann algebras \cite{KVvN}. In this case a locally compact quantum group is given by a von Neumann algebra equipped with a comultiplication and a pair of (left and right invariant) weights. Given a von Neumann   quantum group, its $\C^*$-algebraic version, which fits    the  $\C^*$-system of axioms as formulated in \cite{KV} and \cite{MNW}, may be recovered. Conversely, a $\C^*$-quantum group yields a von Neumann version, making $\C^*$ and von Neumann  approaches equivalent. 

Yet another face of a locally compact quantum group is given by its universal $\C^*$-counterpart \cite{univ} which is directly linked with representation theory of the dual locally compact quantum group. The duality here extends the famous Pontryagin duality discovered in the context of abelian locally compact  groups. In what follows a quantum group will be denoted by $\GG$, its von Neumann algebra by $\Linf(\GG)$, its reduced $\C^*$-algebra by $\C_0(\GG)$ and the universal $\C^*$-algebra by $\C_0^u(\GG)$. 

A   locally compact quantum group can be studied through its representation theory and its actions on $\C^*$-algebras and von Neumann algebras. A distinguished  class of actions is given by taking the quantum quotient  of a locally compact quantum group $\GG$  by its closed quantum subgroup $\HH$. This class  was considered in \cite{Vaes}; it is worth mentioning that   the  von Neumann quotient $\Linf(\GG/\HH)$ can always be easily formed     whereas the existence of the $\C^*$-quotient $\C_0(\GG/\HH)$ is more subtle issue and in general it was proved    under the regularity assumption on $\GG$. If $\HH\subset \GG$ is compact $\C_0(\GG/\HH)$ can always be formed, \cite{embed}. 

In this paper we  study  the actions of locally compact quantum groups  that correspond to idempotent states  (see \cite{SaS}). The latter can be viewed as a generalization of the quantum quotient by a compact quantum subgroup. In particular, an idempotent state $\omega$ on $\GG$ gives rise to a von Neumann coideal $\sN\subset\Linf(\GG)$ which in the subgroup case is the quotient $\Linf(\GG/\HH)$ by a compact quantum subgroups  $\HH\subset \GG$.  We give a von Neumann characterization of $\sN\subset \Linf(\GG)$ corresponding to an idempotent state in terms of the integrability of the $\GG$-action on $\sN$.  We also extend  beyond the coamenable case the characterization of $\C^*$-subalgebras  $\sX\subset \C^u_0(\GG)$ corresponding to compact quantum subgroups   $\HH\subset \GG$ by taking the quotient $\sX  = \C^u_0(\GG/\HH)$ (see \cite{Sal}) and we formulate the von Neumann counterpart of this result. Finally we characterize subalgebras $\sN\subset \Linf(\GG)$ which are of the form $\Linf(\GG/\HH)$ with  $\HH\subset \GG$ being an open quantum subgroup. 
As a byproduct of study   we get a number of universal lifting results for Podle\'s condition, normality and regularity and we generalize a number of results known before to hold  under the coamenability assumption. 

The paper is written as follows. In Section \ref{Prel}  we introduce main definitions and we fix the notation. In Section \ref{lifsec}  we lift some results   that hold for  regular quantum groups from the reduced level to the universal level. As an application we describe  the universal  lift $\C_0^u(\GG/\HH)$ of the $\C^*$-quotient $\C_0(\GG/\HH)$ for a closed quantum subgroup $\HH\subset \GG$. In  Section \ref{integcoid}  a 1-1 correspondence between idempotent states on a locally compact quantum group $\GG$ and integrable coideals in the von Neumann algebra $\Linf(\GG)$ that are preserved by the scaling group is established. Using this result we were able to weaken the assumptions of \cite[Theorem 1]{SaS} and show that  there is a   1-1 correspondence between  idempotent states on $\GG$ and $\psi_\GG$-expected left-invariant von Neumann subalgebras of $\Linf(\GG)$ (see Remark \ref{remphiexp}). 
Section \ref{salres}  is divided into two parts. The first one is the characterization of $\C^*$-subalgebras   $\sX\subset \C_0^u(\GG)$  which are of the form $\C^u_0(\GG/\HH)$  where  $\HH\subset\GG$ is a  compact quantum subgroup. Our techniques are  very similar to those developed in \cite{Sal}  but we were able to drop the  coamenability assumption. In the second part of Section \ref{salres} we characterize von Neumann subalgerbas   $\sN\subset\Linf(\GG)$ which are of the form $\sN = \Linf(\GG/\HH)$ still with $\HH\subset\GG$ being a compact quantum subgroup. In Section  \ref{openidstates}
we establish a 1-1 correspondence between  open quantum subgroups of $\GG$ and   central idempotent states on $\hh\GG$. 
In Section \ref{opencoid}  we characterize   coideal subalgebras  $\sN\subset \Linf(\GG)$ which are of the form   $\sN = \Linf(\GG/\HH)$ for an open quantum subgroup $\HH\subset \GG$.
In the Appendix  we  extend beyond the coamenable case  the result proved in  \cite{DFSW}, stating that a  closed quantum subgroups $\HH\subset \GG$ has a Haagerup property if   $\GG$ has it.
 
\section{Preliminaries}\label{Prel}
We will denote the minimal  tensor product of $\C^*$-algebras with the symbol $\otimes$. The ultraweak tensor product of von Neumann algebras  will be denoted by $\vtens$. For a $\C^*$-subalgebra $\sB$ of a $\C^*$-algebra the multipliers  $\M(\sA)$ of $\sA$, the closed linear span of the set $\big\{ba \st b\in\sB, a\in\sA \big\}$ will be denoted by $\sB\sA$. A morphism  between two $\C^*$-algebras $\sA$ and $\sB$ is a $*$-homomorphism $\pi$ from  $\sA$ into the multiplier algebra $\M(\sB)$,  which is  non-degenerate, i.e $\pi(\sA)\sB = \sB$. We will denote the set of all morphisms from $\sA$ to $\sB$ by $\Mor(\sA,\sB)$. The non-degeneracy of a morphism $\pi$ yields its  natural extension to the unital $*$-homomorphism $\M(\sA)\to \M(\sB)$ also denoted by $\pi$. Let $\sB$ be a $\C^*$-subalgebra of $\M(A)$. We say that $\sB$ is non-degenerate if $\sB\sA = \sA$.  In this case $\M(\sB)$ can  be identified with a $\C^*$-subalgebra of $\M(\sA)$.  The symbol $\flip$ will denote the flip morphism between tensor product of operator algebras. If $X$ is a subset of topological vector space $Y$, by $ X ^{\textrm{\tiny{cls}}}$ we mean the {\it closed linear span} of $X$. In particular if $X\subset \sA$, where $\sA$ is a $\C^*$-algebra then  $ X ^{\textrm{\tiny{norm-cls}}}$ denotes the norm closure of the  linear span of $X$; if   $X\subset \sM$, where $\sM$ is a von Neumann algebra then $ X ^{ \textrm{\tiny{$\sigma$-weak cls}}}$ denotes  the $\sigma$-weak closure of the linear span of $X$.    For a $\C^*$-algebra $\sA$, the space of all functionals on $\sA$ and the state space of $\sA$ will be denoted by $\sA^*$ and $S(\sA)$ respectively. The predual of a von Neumann algebra $\sN$ will be denoted by $\sN_*$. For a Hilbert space $H$ the $\C^*$-algebras of compact   operators on $H$ will be denoted by $\mathcal{K}(H)$. The algebra of bounded operators acting on $H$ will be denoted by  $B(H)$. For $\xi, \eta\in H$, the symbol $\omega_{\xi, \eta}\in B(H)_*$ is the functional $T\mapsto \langle \xi,T\eta\rangle$.  

 For  the theory of locally compact quantum groups we refer to \cite{univ,KV,KVvN}. Let us recall that a von Neumann algebraic locally compact quantum group is a quadruple $\mathbb{G} = (\Linf(\GG), \Delta_\GG,\varphi_\GG,\psi_\GG)$, where $\Linf(\GG)$ is a von Neumann algebra with a coassociative comultiplication $\Delta_\GG\colon\Linf(\GG)\to\Linf(\GG)\vtens\Linf(\GG)$, and $\varphi_\GG$ and $\psi_\GG$ are, respectively,  normal semifinite  faithful left and right Haar weights on $\Linf(\GG)$.  The GNS Hilbert space of the right Haar weight $\psi_\GG$ will be denoted by  $\Ltwo(\GG)$ and the corresponding GNS map will be denoted by $\eta_\GG$. The \emph{antipode}, the \emph{scaling group} and the \emph{unitary antipode} will be denoted by $S$, $(\tau_t)_{t\in \RR}$ and $R$.
We will denote $(\sigma_t)_{t\in\RR}$ and $(\sigma^{\prime}_t)_{t\in\RR}$ the \emph{modular automorphism groups} assigned to $\varphi_\GG$  and $\psi_\GG$ respectively. The following   relation will be used throughout the paper (see \cite[Proposition 6.8]{KV})
\begin{equation}\label{modular}
\Delta_\GG\circ\tau_t=(\sigma_t \otimes \sigma^{\prime}_{-t})\circ\Delta_\GG.
\end{equation}
The multiplicative unitary $\ww^\GG\in B(\Ltwo(\GG)\otimes \Ltwo(\GG))$  is a unique unitary operator such that 
\[\ww^\GG(\eta_\GG(x)\otimes\eta_\GG(y)) = (\eta_\GG \otimes\eta_\GG )(\Delta_\GG(x)(\I\otimes y))\] 
for all $x,y\in D(\eta_\GG)$;  
 $\ww^\GG$  satisfies the pentagonal equation $\ww^\GG_{12}\ww^\GG_{13}\ww^\GG_{23} = \ww^\GG_{23}\ww^\GG_{12}$ \cite{BS,mu}. Using $\ww^\GG$,  $\GG$ can be recovered as follows:
\[\begin{split}
\Linf(\GG) =&\bigl\{ (\omega\otimes\id)\ww^\GG\st\omega\in B(\Ltwo(\GG))_*\bigr\} ^{\textrm{\tiny{$\sigma$-weak cls}}},\\
\Delta_\GG(x) =&\ww^\GG(x\otimes\I){\ww^\GG}^*.
\end{split}\] 
A locally compact quantum group admits a dual object $\hh\GG$. It can be described in terms of  
 $\ww^{\hh\GG} $
\[\begin{split} 
\Linf(\hh\GG)&=\bigl\{( \omega\otimes\id)\ww^{\hh\GG}\st\omega\in B(\Ltwo(\GG))_*\bigr\}^{\textrm{\tiny{$\sigma$-weak cls}}},\\ \Delta_{\hh\GG}(x)&=\ww^{\hh\GG}(x\otimes\I){\ww^{\hh\GG}}^*.
 \end{split}\] where  $\ww^{\hh\GG}=\flip({\ww^{\GG}})^*$. Note that $\ww^\GG\in\Linf(\hh\GG)\bar\otimes\Linf(\GG)$. 
 The  \emph{modular element} of $\GG$ will be denoted by $\delta$.
\begin{definition}\label{von_neumann}
A von Neumann subalgebra $\N$ of $\Linf(\GG)$ is called 
\begin{itemize}
\item \emph{Left coideal}  if $\Delta_\GG(\N)\subset\Linf(\GG)\vtens\N$;
\item \emph{Invariant subalgebra} if $\Delta_\GG(\N)\subset\N\vtens\N$;
\item \emph{Baaj-Vaes subalgebra} if $\N$ is an invariant subalgebra of $\Linf(\GG)$ which is preserved by the unitary antipode $R$  and the scaling group $(\tau_t)_{t\in\RR} $ of $\GG$;
\item \emph{Normal} if $\ww^\GG(\I\otimes\N){\ww^\GG}^*\subset\Linf(\hh\GG)\vtens\N$;
\item \emph{Integrable} if the set of integrable elements with respect to the right Haar weight $\psi_\GG$ is dense in $\N^+$; in other words, the restriction of $\psi_\GG$ to $\N$ is semifinite.
\end{itemize}
\end{definition}
Using terminology of \cite{SaS} a left coideal is nothing but a left-invariant von Neumann subalgebra of $\Linf(\GG)$. In what follows a left coideal will be called a coideal. 
If $\N$ is a coideal of $\Linf(\GG)$, then $\dd{\N}=\N^\prime\cap\Linf(\hh\GG)$ is a   coideal of $\Linf(\hh\GG)$ called the \emph{codual} of $\N$; it turns out that  $\dd{\dd{\N}}=\N$ (see \cite[Theorem 3.9]{embed}). 

The $\C^*$-algebraic version $(\C_0(\GG),\Delta_\GG)$ of a given quantum group $\GG$  is recovered from $\ww^\GG$ as follows 
\[\begin{split}\C_0(\GG) &=\bigl\{(\omega\otimes\id)\ww^\GG\st \omega\in B(\Ltwo(\GG))_*\bigr\}^{\textrm{\tiny{norm-cls}}},\\
\Delta_\GG(x) &=\ww^\GG(x\otimes\I){\ww^\GG}^*.
\end{split}\]
The comultiplication can be viewed as a morphism $\Delta_\GG\in \Mor(\C_0(\GG),\C_0(\GG)\otimes\C_0(\GG))$ and we have  $\ww^\GG\in\M(\C_0(\hh\GG)\otimes \C_0(\GG))$. 

\begin{definition}\label{weakPodandPod}
A non-degenerate $\C^*$-subalgebra  $\sB$ of $ \M(\C_0(\GG))$ 
\begin{itemize}
\item is called \emph{left-invariant}  if $(\mu\otimes\id)\Delta_\GG(\sB) \subset\sB$ for all $\mu\in\C_0(\GG)^*$;
\item is called  \emph{symmetric} if $\ww^\GG(\I\otimes\sB){\ww^\GG}^*\subset\M(\C_0(\hh\GG)\otimes\sB)$;
\item satisfies \emph{Podle\'s condition} if $\Delta_\GG(\sB)(\C_0(\GG)\otimes\I)=\C_0(\GG)\otimes\sB$;
\item satisfies \emph{weak Podle\'s condition} if $(\C_0(\GG)\otimes\I)\Delta_\GG(\sB)(\C_0(\GG)\otimes\I)=\C_0(\GG)\otimes\sB$
 \end{itemize}
 Let us note that Podle\'s condition$\implies$ weak Podle\'s condition$\implies $left-invariance.
  \end{definition}
  
We  adopt  the  following terminology from \cite[Section 1]{SaS}.
\begin{definition}\label{def:expected}
Let $\sB $ be a   $\C^*$-subalgebra of $\C_0(\GG)$. We say that 
\begin{itemize}
\item[(i)] $\sB$ is $\varphi_\GG$-expected if there exists a $\varphi_\GG$-preserving conditional expectation $E$ from $\C_0(\GG)$ onto $\sB$;
\item[(ii)] $\sB$ is $\psi_\GG$-expected if there exists a $\psi_\GG$-preserving conditional expectation $E$ from $\C_0(\GG)$ onto $\sB$;
\item[(iii)] $\sB$ is  expected if there exists a  conditional expectation $E$ from $\C_0(\GG)$ onto $\sB$, which preserves $\psi_\GG$ and $\varphi_\GG$.
\end{itemize}
Let $\sN$ be a   von Neumann subalgebra of $\Linf(\GG)$. We say that 
\begin{itemize}
\item[(i)] $\sN$ is $\varphi_\GG$-expected if there exists a $\varphi_\GG$-preserving conditional expectation $E$ from $\Linf(\GG)$ onto $\sN$;
\item[(ii)] $\sN$ is $\psi_\GG$-expected if there exists a $\psi_\GG$-preserving conditional expectation $E$ from $\Linf(\GG)$ onto $\sN$;
\item[(iii)] $\sN$ is  expected if there exists a  conditional expectation $E$ from $\Linf(\GG)$ onto $\sN$, which preserves $\psi_\GG$ and $\varphi_\GG$.\end{itemize}
\end{definition}
  We will  show (see Proposition \ref{streg}) that  a non-zero $\C^*$-subalgebra $\sB\subset \M(\C_0(\GG))$ such that $(\mu\otimes\id)\Delta_\GG(b) \in\sB$ for all $\mu\in\C_0(\GG)^*$ and $b\in\sB$ is automatically non-degenerate. In particular a non-zero $\C^*$-subalgebra $\sB\subset\M(\C_0(\GG))$ satisfying Podle\'s condition is non-degenerate.  Proposition  \ref{weak} provides a link between  weak Podle\'s condition discussed in \cite[Section 5]{BSV} (called  weak continuity) and  weak Podle\'s condition introduced in Definition \ref{weakPodandPod}. 
Let us recall the definition of an action  of a quantum group  $\GG$.
\begin{definition}\label{def:action_v_C} A \emph{(left) action} of quantum group $\GG$ on a 
\begin{itemize}
\item  von Neumann algebra $\sN$ is a unital injective normal $*$-homomorphism  $\alpha:\sN\to\Linf(\GG)\vtens\sN$ s.t. $(\Delta_\GG\otimes\id)\circ\alpha=(\id\otimes\alpha)\circ\alpha$. 
\item $\C^*$-algebra $\sB$ is an injective  morphism $\alpha\in\Mor(\sB,\C_0(\GG)\otimes\sB)$ s.t.  $(\Delta_\GG\otimes\id)\circ\alpha=(\id\otimes\alpha)\circ\alpha$.  
\end{itemize}
\end{definition}
Let us remark, that    some authors  define $\C^*$-actions as (not necessarily  injective) morphisms  $\alpha\in\Mor(\sB,\C_0(\GG)\otimes\sB)$ satisfying  $(\Delta_\GG\otimes\id)\circ\alpha=(\id\otimes\alpha)\circ\alpha$ and  Podle\'s  condition (see Definition \ref{def:pod_cond}).  For a nice discussion of Podle\'s condition  see \cite{Sol_Pod}. 

In the course of this paper we shall use  the action  $\beta:\C_0(\GG)\to \M(\C_0(\hh\GG) \otimes\C_0(\GG))$  of $\hh\GG$ on $\C_0(\GG)$, where 
\begin{equation}\label{bet}
 \beta(x) ={\ww^\GG}(\I\otimes x){\ww^\GG}^*.
 \end{equation}
 Note that $\beta$ admits a von Neumann extension (which we shall also denote by $\beta$):
\begin{equation}\label{betvN} \beta(x) = \ww^\GG(\I\otimes x){\ww^\GG}^*\in\Linf(\hh\GG)\vtens\Linf(\GG)\end{equation} for all $x\in\Linf(\GG)$. 

\begin{definition}\label{def:pod_cond}
Let $\alpha\in\Mor(\sB,\C_0(\GG)\otimes\sB)$ be an action of $\GG$ on a $\C^*$-algebra $\sB$. We say  that $\sB$ satisfies 
\begin{itemize}
\item \emph{$\alpha$-Podle\'s condition } if $\alpha(\sB)(\C_0(\GG)\otimes\I)=\C_0(\GG)\otimes\sB$,
\item \emph{$\alpha$-weak Podle\'s condition} if 
$\sB=\big\{(\omega\otimes\id)\alpha(\sB)\st \omega\in B(\Ltwo(\GG))_* \big\}^{\text{\tiny{norm-cls}}}$.
\end{itemize}
 Let us note that $\alpha$-Podle\'s condition$\implies$$\alpha$-weak Podle\'s condition.
\end{definition}
\begin{proposition}\label{weak}
Let $\alpha$ be an action of $\GG$ on a $\C^*$-algebra $\sB$. Then $\sB$ satisfies $\alpha$-weak Podle\'s condition if and only if 
\[(\C_0(\GG)\otimes\I)\alpha(\sB)(\C_0(\GG)\otimes\I) = \C_0(\GG)\otimes\sB.\]
\end{proposition}
\begin{proof}
The ``if'' part is clear. In order to get the ``only if `` 
we compute
\begin{align*}
(\C_0(\GG)&\otimes\I)\alpha(\sB)(\C_0(\GG)\otimes\I)\\&= \big\{(\C_0(\GG)\otimes\I)\alpha\big((\omega\otimes\id)(\alpha(\sB))\big)(\C_0(\GG)\otimes\I) \st \omega\in B(\Ltwo(\GG))_*\big\}^{\text{\tiny{norm-cls}}}\\&=\big\{  (\omega\otimes\id\otimes\id)((y\otimes x\otimes \I)\ww^\GG_{12}\alpha(\sB)_{13}\ww^{\GG^*}_{12}(y'\otimes x'\otimes \I)) \st \\& \quad\quad\quad\quad\quad\quad\quad\quad\omega\in B(\Ltwo(\GG))_*,\, y,y'\in\C_0(\hh\GG),\, x,x'\in\C_0(\GG)\big\}^{\text{\tiny{norm-cls}}}\\&=\big\{  (\omega\otimes\id\otimes\id)((y\otimes x\otimes \I) \alpha(\sB)_{13} (y'\otimes x'\otimes \I)) \st \\& \quad\quad\quad\quad\quad\quad\quad\quad\omega\in B(\Ltwo(\GG))_*,\, y,y'\in\C_0(\hh\GG),\, x,x'\in\C_0(\GG)\big\}^{\text{\tiny{norm-cls}}} = \C_0(\GG)\otimes\sB.
\end{align*}
\end{proof}

A locally compact quantum group $\GG$ is assigned with a universal version \cite{univ}. The universal version $\C_0^u(\GG)$  of $\C_0(\GG)$    is equipped  with a comultiplication $\Delta_\GG^u \in\Mor(\C_0^u(\GG),\C_0^u(\GG)\otimes \C_0^u(\GG))$ satisfying   (see \cite[Proposition 6.1]{univ})
\[\Delta_\GG^u(\C_0^u(\GG))(\C_0^u(\GG)\otimes\I) =\C_0^u(\GG)\otimes \C_0^u(\GG) = \Delta_\GG^u(\C_0^u(\GG))(\I\otimes \C_0^u(\GG))\] which will be also referred to as Podle\'s condition. 
The \emph{counit} is a $*$-homomorphism $\varepsilon: \C_0^u(\GG)\to\CC$ satisfying $(\id\otimes\varepsilon)\circ\Delta_\GG^u=\id=(\varepsilon\otimes\id)\circ\Delta_\GG^u$. Multiplicative unitary $\ww^\GG\in\M(\C_0(\hh\GG)\otimes\C_0(\GG))$ admits the universal lift  $\WW^\GG\in\M(\C_0^u(\hh\GG)\otimes \C_0^u(\GG))$.  The reducing morphisms for $\GG$ and $\hh\GG$ will be denoted by $\Lambda_\GG\in\Mor(\C_0^u(\GG),\C_0(\GG))$ and $\Lambda_{\hh\GG}\in\Mor(\C_0^u(\hh\GG),\C_0(\hh\GG))$  respectively. We have
$(\Lambda_{\hh\GG}\otimes\Lambda_\GG)(\WW^\GG)=\ww^\GG$. 
We shall also use the half-lifted versions of $\ww^\GG$, $\Ww^\GG=(\id\otimes\Lambda_\GG)(\WW^\GG)\in\M(\C_0^u(\hh\GG)\otimes \C_0(\GG))$ and $\wW^\GG=(\Lambda_{\hh\GG}\otimes\id)(\WW^\GG )\in\M(\C_0(\hh\GG)\otimes \C_0^u(\GG))$. They satisfy the appropriate versions of pentagonal equation 
\[\begin{split}
\Ww^\GG_{12}\Ww^\GG_{13}\ww^\GG_{23}&=\ww^\GG_{23}\Ww^\GG_{12},\\
\ww^\GG_{12}\wW^\GG_{13}\wW^\GG_{23}&=\wW^\GG_{23}\ww^\GG_{12}.
\end{split}\]
The half-lifted versions of comultiplications will be denoted by  $\Delta_r^{r,u}\in\Mor(\C_0(\GG),\C_0(\GG)\otimes \C_0^u(\GG))$ 
and 
$\hh{\Delta}_r^{r,u}\in\Mor(\C_0(\hh\GG),\C_0(\hh\GG)\otimes \C_0^u(\hh\GG))$, e.g.  
\[\begin{split}
\Delta_r^{r,u}(x)&=\wW^\GG(x\otimes\I){\wW^\GG}^*,~~~~~x\in\C_0(\GG).
\end{split} \]
  We have 
\begin{equation}\label{LDDrL}
\begin{split}
(\Lambda_\GG\otimes\id)\circ\Delta_\GG^u &= \Delta_r^{r,u}\circ\Lambda_\GG,\\
(\Lambda_{\hh\GG}\otimes\id)\circ\Delta_{\hh\GG}^u &= \hh{\Delta}_r^{{r,u}} \circ\Lambda_{\hh\GG}.
\end{split}
\end{equation}
The following forms of  Podle\'s conditions are satisfied 
\begin{align*}
\Delta_r^{r,u}(\C_0(\GG))(\C_0(\GG)\otimes\I)&=\C_0(\GG)\otimes \C_0^u(\GG),\\
\hh{\Delta}_r^{{r,u}}(\C_0(\hh\GG))(\C_0(\hh\GG)\otimes\I)&=\C_0(\hh\GG)\otimes \C_0^u(\hh\GG).
\end{align*}
 We shall consider $\C^*$-subalgebras $\sB$ of $\M(\C_0^u(\GG))$ and the following terminology.
\begin{definition}\label{weakPodandPoduni}
A non-degenerate $\C^*$-subalgebra  $\sB$ of $\M(\C^u_0(\GG))$ 
\begin{itemize}
\item is called \emph{left-invariant}  if $(\mu\otimes\id)\Delta^u_\GG(\sB) \subset\sB$ for all $\mu\in\C_0^u(\GG)^*$;
\item is called  \emph{symmetric} if $\wW^\GG(\I\otimes\sB){\wW^\GG}^*\subset\M(\C_0(\hh\GG)\otimes\sB)$;
\item satisfies \emph{Podle\'s condition} if $\Delta^u_\GG(\sB)(\C^u_0(\GG)\otimes\I)=\C^u_0(\GG)\otimes\sB$;
\item satisfies \emph{weak Podle\'s condition} if $(\C_0^u(\GG)\otimes\I)\Delta_\GG^u(\sB)(\C_0^u(\GG)\otimes\I)=\C_0^u(\GG)\otimes\sB$.
 \end{itemize}
  \end{definition}
  Furthermore we adopt  the following \begin{definition}\label{def:pod_cond1}
Let $\GG$ be a locally compact quantum group, $\sB$ a $\C^*$-algebra and  let $\alpha\in\Mor(\sB,\C^u_0(\GG) \otimes\sB)$ be such that $(\id\otimes\alpha)\circ\alpha = (\Delta_\GG^u\otimes\id)\circ\alpha$.     We say  that $\sB$ satisfies 
\begin{itemize}
\item \emph{$\alpha$-Podle\'s condition } if $\alpha(\sB)(\C^u_0(\GG)\otimes\I)=\C^u_0(\GG)\otimes\sB$,
\item \emph{$\alpha$-weak Podle\'s condition} if 
$\sB=\big\{(\omega\otimes\id)\alpha(\sB)\st \omega\in \C_0^u(\GG)^* \big\}^{\text{\tiny{norm-cls}}}$.
\end{itemize}
 Let us note that $\alpha$-Podle\'s condition$\implies$$\alpha$-weak Podle\'s condition.
\end{definition}
Given a locally compact quantum group $\GG$, the comultiplications $\Delta_\GG$ and $\Delta_\GG^u$ induce Banach algebra structures on $\Linf(\GG)_*$ and $\C_0^u(\GG)^*$ respectively.  The corresponding multiplications will be denoted by $\starr$ and $\staru$. We shall identify $\Linf(\GG)_*$ with a subspace of $\C_0^u(\GG)^*$ when convenient. Under this identification $\Linf(\GG)_*$ forms a two sided ideal in $\C_0^u(\GG)^*$.   Following \cite{univ},  for any $\mu\in\C_0^u(\GG)^*$ we   define a normal map $ \Linf(\GG)\to \Linf(\GG)$ such that $ x\mapsto(\id\otimes\mu)(\wW^\GG(x\otimes\I){\wW^\GG}^*)$ for all $x\in\Linf(\GG)$. We shall use a notation $\mu\staru x=  (\id\otimes\mu)(\wW^\GG(x\otimes\I){\wW^\GG}^*)$.

A state $\omega\in S(\C_0^u(\GG))$ is said to be an {\it idempotent state} if $\omega\staru\omega=\omega$. For a nice survey   describing the history and motivation behind the study of idempotent states see \cite{Salmi_Survey}.  For the theory of idempotent state we refer   to \cite{SaS}.  We shall use \cite[Proposition 4]{SaS} which in particular states that an  idempotent state  $\omega\in S(\C_0^u(\GG))$ is preserved by the universal scaling group $\tau_t^u$:
\begin{equation}\label{prescgr}
\omega\circ\tau_t^u = \omega
\end{equation}
for all $t\in\mathbb{R}$. An idempotent  state $\omega\in S(\C_0^u(\GG))$ yields a conditional expectation $E:\C_0(\GG)\to\C_0(\GG)$ (see \cite{SaS})
\[E(x) =\omega\staru x\] for all $x\in\C_0(\GG)$. Using \eqref{prescgr} we easily get 
\begin{equation}\label{prescgr1}\tau_t(E(x)) = E(\tau_t(x)).\end{equation}
Conditional expectation extends to $E:\Linf(\GG)\to\Linf(\GG)$ and clearly  \eqref{prescgr1} holds for all $x\in\Linf(\GG)$. 
 
 The image of $\sB = E(\C_0(\GG))$ forms a $\C^*$-subalgebra of $\C_0(\GG)$. Let us note that $E$ admits the universal version $E^u:\C_0^u(\GG)\to\C_0^u(\GG)$
\[E^u =(\id\otimes \omega)\circ\Delta_\GG^u.\] In particular $\sB$ admits the universal version $\sB^u = E^u(\C^u_0(\GG))$.  Let $(e_i)_{i\in I}$ be an approximate unit for $\C_0^u(\GG)$. Then  $\lim_i E^u(e_i)=\I$  strictly. Since $E^u(e_i)\in \sB^u$ we see that   $\sB^u$ is a non-degenerate $\C^*$-subalgebra of $\C^u_0(\GG)$. Similarly, $\sB$ is non-degenerate $\C^*$-subalgebra of $\C_0(\GG)$. It is easy to check that 
\[\M(\sB) = \{x\in\M(\C_0(\GG)):E(x) = x\}\] and 
\begin{equation}\label{eq:ce_mult}\M(\sB^u) = \{x\in\M(\C^u_0(\GG)):E^u(x) = x\}.\end{equation}
 Since $E$ is $\Delta_\GG$ - covariant, i.e.
\[\Delta_\GG\circ E = (\id\otimes E)\circ\Delta_\GG\] we conclude that
\[\Delta_\GG(\sB)\subset \M(\C_0(\GG)\otimes\sB).\] Moreover $\sB$ satisfies Podle\'s condition. Indeed
\[\Delta_\GG(\sB)(\C_0(\GG)\otimes \I) = (\id\otimes E)(\Delta_\GG(\C_0(\GG))(\C_0(\GG)\otimes\I)) = \C_0(\GG)\otimes \sB.\] Similarly, $\sB^u$ satisfies Podle\'s condition. 

A locally compact quantum group $\GG$ is called  \emph{coamenable} if $\Lambda_\GG\in\Mor(\C_0^u(\GG),\C_0(\GG))$ is an isomorphism. In this case we shall identify  $\C_0^u(\GG)=\C_0(\GG)$. It can be shown that $\GG$ is  coamenable if and only if $\C_0(\GG)$ admits  counit   \cite[Theorem 3.1]{BT}.
A quantum group $\GG$ is  \emph{compact} if the $\C^*$-algebra $\C_0(\GG)$ is unital. In this case we write $\C(\GG)$ and $\C^u(\GG)$ instead of $\C_0(\GG)$ and $\C_0^u(\GG)$ respectively.

A locally compact quantum group $\GG$ is said to be
\begin{itemize}
\item  \emph{regular}  if 
\begin{equation}\label{regularity}
\big\{ (\id\otimes\omega)(\Sigma\ww^\GG) \st \omega\in B(\Ltwo(\GG))_* \big\}^{\text{\tiny{norm-cls}}} =\mathcal{K}(\Ltwo(\GG)),
\end{equation}
\item \emph{semi-regular} if 
\[\big\{ (\id\otimes\omega)(\Sigma\ww^\GG) \st \omega\in B(\Ltwo(\GG))_* \big\}^{\text{\tiny{norm-cls}}} \supset\mathcal{K}(\Ltwo(\GG))\] 
\end{itemize} where $\Sigma:\Ltwo(\GG)\otimes\Ltwo(\GG) \to \Ltwo(\GG)\otimes\Ltwo(\GG)$ is the Hilbert space flip.  
It can be shown that    $\GG$ is  regular if and only if  (see \cite[Proposition 3.6]{BS})
\begin{equation}\label{reg}
\C_0(\hh\GG)\otimes \C_0(\GG)=\bigl\{(x\otimes \I)\ww^\GG(\I\otimes y)\st x\in \C_0(\hh\GG), y\in \C_0(\GG)\bigr\}^{\text{\tiny{norm-cls}}}.
\end{equation}
Let us note that \eqref{reg} holds if and only if    $\C_0(\GG)$ satisfies  $\beta$-Podle\'s condition (for the definition of $\beta$ see \eqref{bet}). Remarkably, it was proved in  \cite[Proposition 5.6]{BSV} that $\GG$ is regular if and only if $\C_0(\GG)$ satisfies $\beta$-weak Podle\'s condition. 

Let $\HH$ and $\GG$ be  locally compact quantum groups. Then a morphism $\pi\in\Mor(\C_0^u(\GG),\C_0^u(\HH))$ such that 
\[
(\pi\otimes\pi)\circ\Delta_\GG^u=\Delta_\HH^u\circ\pi
\] is said to define a homomorphism from $\HH$ to $\GG$. If $\pi(\C_0^u(\GG))=\C_0^u(\HH)$, then $\HH$ is called  \emph{Woronowicz-closed  quantum subgroup} of $\GG$ \cite{DKSS}. A homomorphism from $\HH$ to $\GG$ admits the dual homomorphism  ${\hh\pi}\in\Mor(\C_0^u(\hh\HH), \C_0^u(\hh\GG))$ such that 
\[
(\id\otimes\pi)(\WW^\GG) = ({\hh\pi}\otimes\id)(\WW^\HH).
\]
A homomorphism  from  $\HH$ to $\GG$ identifies $\HH$ as a  \emph{closed quantum subgroup} $\GG$ if there exists an injective  normal unital $*$-homomorphism $\gamma:\Linf(\hh\HH)\to\Linf(\hh\GG)$ such that 
\[
\Lambda_{\hh\GG}\circ{\hh\pi}(x) = \gamma\circ\Lambda_{\hh\HH}(x)
\] for all $x\in\C_0^u(\hh\HH)$. 
Let $\HH$ be a closed quantum subgroup of $\GG$, then $\HH$ acts on $\Linf(\GG)$   (in the von Neumann algebraic sense) by the following formula 
\[
 \alpha:\Linf(\GG)\to\Linf(\GG)\vtens\Linf(\HH),~~~ x\mapsto V(x\otimes \I)V^*
 \]
where \begin{equation}\label{eq:def_V}V =(\gamma\otimes\id)(\ww^\HH).\end{equation} The fixed point space of $\alpha$  is denoted by  \[\Linf(\GG/\HH) =\big\{x\in\Linf(\GG)\st \alpha(x)=x\otimes\I \big\}\]  and referred to as the algebra of bounded functions on the quantum homogeneous space $\GG/\HH$. If $\HH$ is  a compact quantum subgroup of $\GG$ then there is a conditional expectation $E:\Linf(\GG)\to\Linf(\GG)$ onto  $\Linf(\GG/\HH)$ which is defined by 
\begin{equation}\label{eq:def_ce_H} E=(\id\otimes\psi_\HH)\circ\alpha\end{equation} where $\psi_\HH$ is the Haar measure of $\HH$. 

According to [9, Definition 2.2] we
say that $\HH$ is an  \emph{open quantum subgroup} of $\GG$ if there is a surjective normal $*$-homomorphism
$\rho:\Linf(\GG)\to\Linf(\HH)$  such that 
\[\Delta_\HH \circ\rho= (\rho\otimes\rho)\circ\Delta_\GG.
\]
Every open quantum subgroup is closed \cite[Theorem 3.6]{KKS}. We recall that a projection $P\in\Linf(\GG)$ is a  \emph{group-like projection} if $\Delta_\GG(P)(\I\otimes P)=P\otimes P$. There is a 1-1 correspondence between (isomorphism  classes of) open quantum subgroups of $\GG$ and  central group-like projections in $\GG$ \cite[Theorem 4.3]{KKS}. The group-like projection assigned to $\HH$, i.e.  the central support of $\rho$, will be denoted by $\I_\HH$. 
\section{Lifting and non-degeneracy results}\label{lifsec}
In this section we shall prove a number of lifting and non-degeneracy results. Not all of them will be used in next sections. The reader who is focused on the theory of idempotent states should get familiar with  Proposition \ref{streg}, Proposition \ref{streg1} and  Proposition \ref{sym}  and may go directly to next sections but we believe that the results obtained in this section  are interesting in its own and have the potential to be used elsewhere. 

Let $\sB$ be a non-zero $\C^*$-subalgebra of $\M(\C_0(\GG))$. If $\sB$ is non-degenerate then  $\M(\sB)$ can  be identified with   a subalgebra of $\M(\C_0(\GG))$. The next  proposition is thus important while considering the symmetry  condition $\ww^\GG(\I\otimes \sB){\ww^\GG}^*\subset \M(\C_0(\hh\GG)\otimes \sB)$ for a non-zero left-invariant $\sB\subset\M(\C_0(\GG))$ (see  Section \ref{salres}). 
 
\begin{proposition}\label{streg}
Let $\sB$ be a non-zero left-invariant  $\C^*$-subalgebra of $\M(\C_0(\GG))$.  Then  $\sB$ is non-degenerate, $\sB\C_0(\GG) = \C_0(\GG)$. 
\end{proposition}
\begin{proof}
Let us define \[\sX=\textrm{span}\bigl\{b a\st b\in \sB, a\in \C_0(\GG) \bigr\}^{\text{\tiny{$\sigma$-weak cls}}}.\]  $\sX$ is a $\sigma$-weakly closed right ideal of $\Linf(\GG)$, so there exists a projection $p\in\Linf(\GG)$ such that $\sX=p\Linf(\GG)$.

Let $b\in \sB$ and $a\in\C_0(\GG)$. We shall prove that for all $\mu\in\B(\Ltwo(\GG))_*$, $(\mu\otimes\id)(\Delta_\GG(ba))\in \sX$.  It suffices to  check the latter for  $\mu = c\cdot \omega$ where $c\in\C_0(\GG)$ and $\omega\in\B(\Ltwo(\GG))_*$. Let us note that 
\[ 
(c\cdot \omega\otimes\id)(\Delta_\GG(ba))   = (\omega\otimes \id)(\Delta_\GG(b)\Delta_\GG(a)(c\otimes\I)) 
\]
and 
\[(\omega\otimes \id)(\Delta_\GG(b)\Delta_\GG(a)(c\otimes\I))\in\big\{(\mu\otimes\id)(\Delta_\GG(b))|b\in\sB, \mu\in\C_0(\GG)^*\big\}^{\text{\tiny{norm-cls}}}\C_0(\GG)\subset\sX.\] 
Thus we conclude that $\Delta_\GG(\sX)\subset \Linf(\GG)\vtens\sX$.  In particular  $\Delta_\GG(p)\leq \I\otimes p$. Using  \cite[Lemma 6.4]{KV} we get $p=\I$ and $\sX=\Linf(\GG)$. Therefore
\begin{align*}
\sB \C_0(\GG) &\supseteq \bigl\{(\omega\otimes\id)(\ww^\GG(b\otimes\I){\ww^\GG}^*)a\st \omega\in B(\Ltwo(\GG))_*, b\in\sB, a\in \C_0(\GG)\bigr\}^{\text{\tiny{norm-cls}}}\\
&=  \bigl\{(c\cdot\omega\otimes\id)(\ww^\GG(b\otimes\I){\ww^\GG}^*)a\st \omega\in B(\Ltwo(\GG))_*, b\in\sB, a\in\C_0(\GG), c\in \C_0(\hh\GG)\bigr\}^{\text{\tiny{norm-cls}}}\\
&= \bigl\{(\omega\otimes\id)(\ww^\GG(b\otimes\I))\st \omega\in B(\Ltwo(\GG))_*, b\in\sB \bigr\}^{\text{\tiny{norm-cls}}}\C_0(\GG)\\&= \bigl\{(\omega\otimes\id)(\ww^\GG(ba\otimes\I))\st \omega\in B(\Ltwo(\GG))_*, b\in\sB, a\in \C_0(\GG)\bigr\}^{\text{\tiny{norm-cls}}} \C_0(\GG)\\&= \C_0(\GG) \C_0(\GG)   = \C_0(\GG)  
\end{align*}
where  the second last equality holds because the von Neumann algebra generated by $\sB \C_0(\GG)$ is $\Linf(\GG)$.
\end{proof}
In the next proposition we will prove the counterpart of Proposition \ref{streg} for $\C^*$-subalgebras of $\M(\C_0^u(\GG))$.
\begin{proposition}\label{streg1}
Let $\sB$ be a left-invariant  $\C^*$-subalgebra of $\M(\C^u_0(\GG))$ such that $\Lambda_\GG(\sB)\neq 0$.  Then  $\sB$ is non-degenerate, $\sB \C^u_0(\GG) = \C^u_0(\GG)$. 
\end{proposition}
\begin{proof}
Let us consider $\sD = \Lambda_\GG(\sB)\subset\M(\C_0(\GG))$. Then it is easy to check that $\sD$ satisfies the assumptions of Proposition \ref{streg}. In particular $\sD \C_0(\GG) = \C_0(\GG)$. We compute 
\begin{align*}
\sB \C^u_0(\GG) &\supseteq \bigl\{(\omega\otimes\id)(\wW^\GG(d\otimes\I){\wW^\GG}^*)a\st \omega\in B(\Ltwo(\GG))_*, d\in\sD, a\in \C^u_0(\GG)\bigr\}^{\text{\tiny{norm-cls}}}\\
&=  \bigl\{(c\cdot\omega\otimes\id)(\wW^\GG(d\otimes\I){\wW^\GG}^*)a\st \omega\in B(\Ltwo(\GG))_*, d\in\sD, a\in\C^u_0(\GG), c\in \C_0(\hh\GG)\bigr\}^{\text{\tiny{norm-cls}}}\\
&= \bigl\{(\omega\otimes\id)(\wW^\GG(d\otimes\I))\st \omega\in B(\Ltwo(\GG))_*, d\in\sD \bigr\}^{\text{\tiny{norm-cls}}}\C^u_0(\GG)\\&= \bigl\{(\omega\otimes\id)(\wW^\GG(da\otimes\I))\st \omega\in B(\Ltwo(\GG))_*, d\in\sD, a\in \C_0(\GG)\bigr\}^{\text{\tiny{norm-cls}}} \C^u_0(\GG)\\&= \C^u_0(\GG)\C^u_0(\GG)   = \C^u_0(\GG)  
\end{align*}
where in the second last equality we used $\sD\C_0(\GG) = \C_0(\GG)$. 
\end{proof}
Let us emphasize that the assumption $\Lambda_\GG(\sB)\neq 0$ in Proposition \ref{streg1} is essential. In order to see the relevant example let us consider a non-coamenable locally compact quantum group $\GG$ and $\sB = \ker\Lambda_\GG\subset\C_0^u(\GG)$. Then $\sB\subset\C_0^u(\GG)$ is a $\C^*$-algebra which is left-invariant. Indeed   for all $b\in\sB$ and $\mu\in\C_0^u(\GG)^*$ we have 
\[
\begin{split}
\Lambda_\GG((\mu\otimes\id)(\Delta^u_\GG(b))) &= (\mu\otimes\id)\big((\id\otimes\Lambda_\GG)\Delta^u_\GG(b)\big)\\
&= (\mu\otimes\id)( \Delta^{u,r}_r (\Lambda_\GG(b))) = 0
\end{split}
\] thus $(\mu\otimes\id)(\Delta^u_\GG(b))\in\sB$. But $\sB$ being an ideal in $\C_0^u(\GG)$ cannot be non-degenerate.

Let $\omega\in S(\C_0^u(\GG))$ be  an idempotent state  on $\GG$. It  yields a pair of conditional expectations $E^u:\C_0^u(\GG)\to\C_0^u(\GG)$ and $E:\C_0(\GG)\to\C_0(\GG)$ and $\C^*$-subalgebras $ \sX^u=E^u(\C_0^u(\GG))\subset\C_0^u(\GG)$ and  $\sX =E (\C_0 (\GG))\subset\C_0 (\GG)$ satisfying Podle\'s conditions. Note that $\Lambda_\GG(\sX^u) = \sX$. In what follows we shall analyze the passage from the reduced level to the universal level in a more general context of a $\C^*$-subalgebra  $\sB\subset\M(\C_0(\GG))$ satisfying (weak) Podle\'s condition. In order to do this we need preparatory results. 

Let us recall that $\GG$ is a    regular quantum group if and only if \eqref{reg} is satisfied. We shall  show that equality \eqref{reg} implies apparently stronger condition with  "half-lifted" multiplicative unitary.
\begin{lemma}\label{huni}
A locally compact quantum group $\GG$ is regular if and only if 
\[
\C_0(\hh\GG)\otimes \C_0^u(\GG)=\bigl\{(x\otimes \I)\wW^\GG(\I\otimes y)\st x\in \C_0(\hh\GG), y\in \C_0^u(\GG)\bigr\}^{\text{\tiny{norm-cls}}}.
\]
\end{lemma}
\begin{proof}
The "if" part is trivial. To prove the "only if", assume $\GG$ is a regular locally compact quantum group. 
Applying $(\id\otimes\Delta_r^{r,u})$ to \eqref{reg} we obtain
\begin{equation}\begin{split}\label{uni}
\C_0(\hh\GG)\otimes\Delta_r^{r,u}(\C_0(\GG)) =\bigl\{(x\otimes\I\otimes\I)\ww^\GG_{12}\wW^\GG_{13}(\I\otimes\Delta_r^{r,u}(y))~|~ x\in \C_0(\hh\GG), y\in \C_0(\GG)\bigr\}^{\text{\tiny{norm-cls}}}.
\end{split}
\end{equation}
Since  $\Delta_r^{r,u}$ satisfies Podle\'s condition,  slicing the second leg of \eqref{uni} we get 
\begin{align*}
\C_0(\hh\GG)&\otimes \C_0^u(\GG)=\\
&= \bigl\{(\id\otimes\mu\otimes\id)(x\otimes\I\otimes\I)\ww^\GG_{12}\wW^\GG_{13}(\I\otimes \Delta_r^{r,u}(c))(\I\otimes d\otimes\I))\st \\&   \quad\quad \quad \quad\quad\quad \quad \quad \quad\quad \quad \quad 
\mu\in\C_0(\GG)^*, x\in\C_0(\hh\GG), c,d\in\C_0(\GG)\bigr\}^{\text{\tiny{norm-cls}}}\\
&= \bigl\{(\id\otimes\mu\otimes\id)(x\otimes\I\otimes\I)\ww^\GG_{12}\wW^\GG_{13}(\I\otimes e\otimes y)\st\\&   \quad\quad \quad \quad\quad\quad \quad \quad \quad\quad \quad \quad 
\mu\in\C_0(\GG)^*, x\in\C_0(\hh\GG), e\in\C_0(\GG), y\in\C_0^u(\GG)\bigr\}^{\text{\tiny{norm-cls}}} \\
&=  \bigl\{(\id\otimes\mu\otimes\id)\bigr((x\otimes \I)\ww^\GG(1\otimes e)\otimes\I\bigl)\wW^\GG_{13}(\I\otimes\I\otimes y)\st\\&  \quad\quad \quad \quad\quad\quad \quad \quad \quad\quad \quad \quad  
\mu\in\C_0(\GG)^*, x\in\C_0(\hh\GG), e\in\C_0(\GG), y\in\C_0^u(\GG) \bigr\}^{\text{\tiny{norm-cls}}}\\
& =   \bigl\{(\id\otimes\mu\otimes\id)(x\otimes a\otimes \I)\wW^\GG_{13}(\I\otimes\I\otimes y)\st\\&   \quad\quad \quad \quad\quad\quad \quad \quad \quad\quad \quad \quad 
\mu\in\C_0(\GG)^*, x\in\C_0(\hh\GG), a\in\C_0(\GG), y\in\C_0^u(\GG) \bigr\}^{\text{\tiny{norm-cls}}} \\
&=  \bigl\{ (x\otimes \I)\wW^\GG(\I\otimes y)\st x\in\C_0(\hh\GG), y\in\C_0^u(\GG)\bigr\}^{\text{\tiny{norm-cls}}}.
\end{align*}
\end{proof}
We shall show that  yet stronger condition   with the universal bicharacter $\WW^\GG$ is also implied by \eqref{reg}. 
\begin{lemma}
A locally compact quantum group $\GG$ is regular if and only if 
\[\C^u_0(\hh\GG)\otimes \C_0^u(\GG)=\bigl\{(x\otimes \I)\WW^\GG(\I\otimes y)\st x\in \C^u_0(\hh\GG), y\in \C_0^u(\GG)\bigr\}^{\text{\tiny{norm-cls}}}.
\]
\end{lemma}
\begin{proof}
The "if" part is trivial. In order to prove the "only if"  assume $\GG$ is a regular. Lemma \ref{huni} yields 
\begin{equation}\label{haluni} 
\C_0(\hh\GG)\otimes \C_0^u(\GG)=\bigl\{(x\otimes \I)\wW^\GG(\I\otimes y)\st x\in \C_0(\hh\GG), y\in \C_0^u(\GG)\bigr\}^{\text{\tiny{norm-cls}}}.
\end{equation}
Applying $(\hh\Delta_r^{r,u}\otimes\id)$ to \eqref{haluni} we get 
\begin{equation}\label{unireg}
\hh\Delta_r^{r,u}(\C_0(\hh\GG))\otimes \C^u_0(\GG) =\bigl\{(\hh\Delta_r^{r,u}(x)\otimes\I)\WW^\GG_{23}\wW^\GG_{13}(\I\otimes\I\otimes  y)\st x\in \C_0(\hh\GG), y\in \C_0^u(\GG)\bigr\}^{\text{\tiny{norm-cls}}}.
\end{equation}
Since  $\hh\Delta_r^{r,u}$ satisfies Podle\'s condition, by slicing the first leg of \eqref{unireg} we get 
\begin{align*}
\C^u_0(\hh\GG)\otimes \C_0^u(\GG)&=
\bigl\{(\mu\otimes \id \otimes\id)(\hh\Delta_r^{r,u}(x)\otimes\I)\WW^\GG_{23}\wW^\GG_{13}(\I\otimes\I\otimes  y)\st\\&  
 \quad \quad\quad\quad \quad \quad \quad\quad \quad \quad 
\mu\in\B(\Ltwo(\GG))_*, x\in \C_0(\hh\GG), y\in \C^u_0(\GG)\bigr\}^{\text{\tiny{norm-cls}}}\\
&=\bigl\{(\mu\otimes \id \otimes\id)(((a\otimes\I)\hh\Delta_r^{r,u}(b))\otimes\I)\WW^\GG_{23}\wW^\GG_{13}(\I\otimes\I\otimes  y)~|\\&  
 \quad\quad \quad\quad \quad \quad \quad\quad \quad \quad 
\mu\in\B(\Ltwo(\GG))_*, a,b\in \C_0(\hh\GG), y\in \C^u_0(\GG)\bigr\}^{\text{\tiny{norm-cls}}}\\
&=\bigl\{(\mu\otimes \id \otimes\id)(a\otimes b\otimes\I)\WW^\GG_{23}\wW^\GG_{13}(\I\otimes\I\otimes  y)\st \\&  
 \quad\quad \quad\quad \quad \quad \quad\quad \quad \quad 
\mu\in\B(\Ltwo(\GG))_*, a\in \C_0(\hh\GG),b\in \C^u_0(\hh\GG), y\in \C_0^u(\GG)\bigr\}^{\text{\tiny{norm-cls}}}\\
&=\bigl\{(\mu\otimes \id \otimes\id)(a\otimes b\otimes\I)\WW^\GG_{23} (\I\otimes\I\otimes  y)~|\\&  
 \quad\quad \quad \quad \quad \quad \quad\quad \quad \quad 
\mu\in\B(\Ltwo(\GG))_*, a\in \C_0(\hh\GG),b\in \C^u_0(\hh\GG), y\in\C_0^u(\GG)\bigr\}^{\text{\tiny{norm-cls}}} \\
&=\bigl\{ (x\otimes\I)\WW^\GG (\I\otimes  y)\st  x\in \C^u_0(\hh\GG), y\in \C_0^u(\GG)\bigr\}^{\text{\tiny{norm-cls}}} 
\end{align*}
where in the fourth  equality we used Lemma \ref{huni}. 
\end{proof}
Let us  prove an auxiliary lemma.
\begin{lemma}\label{weakPodreg}
Let $\GG$ be a regular locally compact quantum group, $\sD$   a $\C^*$-algebra and $\alpha\in\Mor(\sD,\C^u_0(\GG)\otimes\sD)$  a morphism satisfying  
\[(\id\otimes\alpha)\circ\alpha = (\Delta_\GG^u\otimes\id)\circ\alpha \] If $\sD$ satisfies $\alpha$-weak Podle\'s condition 
then it satisfies $\alpha$-Podle\'s condition.
\end{lemma}
\begin{proof}In what follows we shall denote $\alpha_r = (\Lambda_\GG\otimes\id)\circ\alpha$. Clearly, 
$\alpha$-weak Podle\'s condition yields 
\[\sD = \big\{(\omega\otimes\id)(\alpha_r(d))\st d\in\sD, \omega\in \Linf(\GG)_*\big\}^{\text{\tiny{norm-cls}}}.\]We compute 
\[\begin{split}
(\C_0^u(\GG)&\otimes\I)\alpha(\sD) =
\big\{(a\otimes\I)\alpha((\omega\otimes\id)\alpha_r(d))\st a\in\C_0^u(\GG),\omega\in\Linf(\GG)_*,d\in\sD\big\}^{\text{\tiny{norm-cls}}}\\&=
\big\{(\omega\otimes\id\otimes\id)\big(((\I\otimes a)\wW^\GG)_{12}\alpha_r(d)_{13}\wW^{\GG^*}_{12}\big) 
\st a\in\C_0^u(\GG),\omega\in\Linf(\GG)_*,d\in\sD\big\}^{\text{\tiny{norm-cls}}}\\&=
\big\{(\omega\otimes\id\otimes\id)\big(\alpha_r(d)_{13}((\I\otimes a){\wW^\GG}^*(b\otimes\I))_{12}\big)\st
\\& \quad\quad\quad\quad\quad\quad\quad\quad a\in\C_0^u(\GG),b\in\C_0(\hh\GG), \omega\in\Linf(\GG)_*,d\in\sD\big\}^{\text{\tiny{norm-cls}}} \\&=\C_0^u(\GG)\otimes \sD.
\end{split}\]
where in the fourth equality we use Lemma \ref{huni}. 
\end{proof}
In the  next theorem  we shall show that a  $\C^*$-subalgebra  $\sB$ of $\M(\C_0(\GG))$ satisfying    Podle\'s condition admits a unique universal lift under  regularity condition on  $\GG$. Then we shall discuss the universal lift for $\sB\subset \M(\C_0(\GG))$ satisfying weak Podle\'s condition (with regularity condition dropped).
\begin{theorem}\label{lift}
Let $\GG$ be a regular locally compact quantum group. Suppose that $\sB$ is a non-zero $\C^*$-subalgebra of
 $\M(\C_0(\GG))$  satisfying  Podle\'s condition. Then there 
 exists a unique $\C^*$-subalgebra $\sB^u\subset\M(\C_0^u(\GG))$ such that $\Lambda_\GG(\sB^u)=\sB$ and $
 \Delta_\GG^u(\sB^u)(\C_0^u(\GG)\otimes \I)=\C_0^u(\GG)\otimes \sB^u$.
\end{theorem}
\begin{proof}
Let us  define \[\sB^u=\bigl\{(\omega\otimes\id)(\wW^\GG(b\otimes\I){\wW^\GG}^*)\st \omega\in B(\Ltwo(\GG))_*, b\in B\bigr\}^{\text{\tiny{norm-cls}}}\subset \M(\C_0^u(\GG)).\] Clearly  $\sB^u$ is a  $*$-closed subspace of $\M(\C_0^u(\GG))$. We will show that $\sB^u$ satisfies all the required conditions. We shall first  check the Podle\'s condition:  $\Delta_\GG^u(\sB^u)(\C_0^u(\GG)\otimes \I)=\C_0^u(\GG)\otimes \sB^u$. Let us note that  we abuse here the terminology concerning Podle\'s condition since we do not know yet if $\sB^u$ forms a $\C^*$-algebra. This will be checked later. We compute 
\begin{align*}
\Delta_\GG^u(\sB^u)(\C_0^u(\GG)\otimes \I)&=  \bigl\{ \Delta_\GG^u\bigr((\omega\otimes\id)(\wW^\GG(b\otimes \I){\wW^\GG}^*)\bigl)(y\otimes \I)\st\\&\quad\quad\quad\quad\quad\quad\quad\quad\quad\quad 
\omega\in B(\Ltwo(\GG))_*, b\in \sB, y\in\C_0^u(\GG)\bigr\}^{\text{\tiny{norm-cls}}}\\
&= \bigl\{(\omega\otimes\id\otimes\id)\bigl(\wW^\GG_{12}\wW^\GG_{13}(b\otimes\I\otimes\I)\wW^{\GG^*}_{13}\wW^{\GG^*}_{12}\bigr)(\I\otimes y \otimes \I)\st\\&
\quad\quad\quad\quad\quad\quad\quad\quad\quad\quad \omega\in B(\Ltwo(\GG))_*, b\in \sB, y\in\C_0^u(\GG)\bigr\}^{\text{\tiny{norm-cls}}}\\
&= \bigl\{(\omega\otimes\id\otimes\id)\bigl(\wW^\GG_{12}\wW^\GG_{13}(b\otimes\I\otimes\I)\wW^{\GG^*}_{13}\wW^{\GG^*}_{12}\bigr)(c\otimes y \otimes \I)\st\\&
\quad\quad\quad\quad\quad\quad\quad\quad\quad\quad
\omega\in B(\Ltwo(\GG))_*, b\in \sB, y\in\C_0^u(\GG), c\in\C_0(\hh\GG)\bigr\}^{\text{\tiny{norm-cls}}}\\
&=  \bigl\{ (c\cdot\omega\otimes\id\otimes\id)\bigl(\wW^\GG_{12}\wW^\GG_{13}(b\otimes\I\otimes\I)\wW^{\GG^*}_{13}\bigr)(\I\otimes y \otimes \I)\st\\&
\quad\quad\quad\quad\quad\quad\quad\quad\quad\quad
\omega\in B(\Ltwo(\GG))_*, b\in \sB, y\in\C_0^u(\GG), c\in\C_0(\hh\GG)\bigr\}^{\text{\tiny{norm-cls}}}\\
&= \bigl\{ (\omega\otimes\id\otimes\id)\bigl((d\otimes\I\otimes\I)\wW^\GG_{12}(\I\otimes y\otimes \I)\bigr)\bigl(\wW^\GG_{13}(b\otimes\I\otimes\I)\wW^{\GG^*}_{13}\bigr)\st \\&
\quad\quad\quad\quad\quad\quad\quad\quad\quad\quad 
\omega\in B(\Ltwo(\GG))_*,d\in\C_0(\hh\GG), y\in\C_0^u(\GG), b\in \sB\bigr\}^{\text{\tiny{norm-cls}}}\\
&=  \bigl\{(\omega\cdot d\otimes\id\otimes\id)\bigl((\I\otimes y\otimes \I)\bigr)\bigl(\wW^\GG_{13}(b\otimes\I\otimes\I)\wW^{\GG^*}_{13}\bigr)\st\\&
\quad\quad\quad\quad\quad\quad\quad\quad\quad\quad
\omega\in B(\Ltwo(\GG))_*,d\in\C_0(\hh\GG), y\in\C_0^u(\GG), b\in \sB \bigr\}^{\text{\tiny{norm-cls}}}\\
&=  \bigl\{(\omega\otimes\id\otimes\id)(\I\otimes y\otimes \I)\bigl(\wW^\GG_{13}(b\otimes\I\otimes\I)\wW^{\GG^*}_{13}\bigr)\st\\&
\quad\quad\quad\quad\quad\quad\quad\quad\quad\quad
\omega\in B(\Ltwo(\GG))_*, y\in\C_0^u(\GG), b\in \sB\bigr\}^{\text{\tiny{norm-cls}}}\\
&=  \bigl\{(y\otimes \I)\bigl((\omega\otimes\id\otimes\id)\wW^\GG_{13}(b\otimes\I\otimes\I)\wW^{\GG^*}_{13}\bigr)\st\\&
\quad\quad\quad\quad\quad\quad\quad\quad\quad\quad 
y\in\C_0^u(\GG), \omega\in B(\Ltwo(\GG))_*, b\in \sB \bigr\}^{\text{\tiny{norm-cls}}} 
 = \C_0^u(\GG)\otimes \sB^u
\end{align*} where in the sixth equality we used Lemma \ref{huni}.

Let us check  that $\Lambda_\GG(\sB^u)=\sB$:
\begin{align*}
\sB &=\bigl\{(\omega\otimes\id)(\Delta_\GG(b)(c\otimes\I))\st \omega\in B(\Ltwo(\GG))_*, b\in \sB, c\in\C_0(\GG)\bigr\}^{\text{\tiny{norm-cls}}}\\
&=\bigl\{(c\cdot\omega\otimes\id)(\ww^\GG(b\otimes\I) {\ww^\GG}^*)\st \omega\in B(\Ltwo(\GG))_*, b\in \sB, c\in\C_0(\GG)\bigr\}^{\text{\tiny{norm-cls}}}\\
&=\bigl\{(\omega\otimes\id)(\id\otimes\Lambda_\GG){\wW^\GG}(b\otimes\I){\wW^\GG}^*\st \omega\in B(\Ltwo(\GG))_*, b\in \sB\bigr\}^{\text{\tiny{norm-cls}}}\\
&=\bigl\{\Lambda_\GG\bigl((\omega\otimes\id){\wW^\GG}(b\otimes\I){\wW^\GG}^*\bigr)\st \omega\in B(\Ltwo(\GG))_*, b\in \sB\bigr\}^{\text{\tiny{norm-cls}}}=\Lambda_\GG(\sB^u).
\end{align*}
Let us prove that $\sB^u$ is a $\C^*$-algebra. 
By applying $\Lambda_\GG\otimes\id$ to Podle\'s condition  satisfied by $\sB^u$ we have 
$\Delta_r^{r,u}(\sB)(\C_0(\GG)\otimes\I)=\C_0(\GG)\otimes \sB^u$. Taking the conjugate of this equality we obtain $\C_0(\GG)\otimes \sB^u=(\C_0(\GG)\otimes\I)\Delta_r^{r,u}(\sB)$. Let us fix $\epsilon>0$ and  $x,x^\prime\in \sB^u$, $0\neq a\in \C_0(\GG)$. There exist finite sets $I$ and $J$ such that 
\begin{align*}
\| a\otimes x - \sum_{i\in I}(y_i\otimes\I)\Delta_r^{r,u}(x_i)\| &\leq \epsilon\\
\| a^*\otimes x^\prime - \sum_{j\in J}\Delta_r^{r,u}(x^\prime_j)(y_j\otimes\I)\| &\leq \epsilon
\end{align*}
 for some $x_i,x^\prime_j,y_i, y^\prime_j\in\C_0(\GG)$ so
\begin{equation}\label{aprox}
\|aa^*\otimes xx^\prime - \sum_{i,j}(y_i\otimes\I)\Delta_r^{r,u}(x_ix_j^\prime)(y_j^\prime\otimes\I)\| \leq \epsilon^\prime
\end{equation}
where $\epsilon^\prime = (\|a\otimes x\|+\|a^*\otimes x^\prime\|)\epsilon$. 
Now choosing a functional  $\mu\in\C_0^u(\GG)^*$ such that $\mu(aa^*)= 1$ and   applying  $\mu\otimes \id$ to \eqref{aprox} we get 
$$\|xx^\prime - \sum_{i,j}(y_j^\prime\cdot\mu\cdot y_i\otimes \I)(\Delta_r^{r,u}(x_ix^\prime_j))\|\leq\|\mu\|\epsilon^\prime.$$
This  shows that $\sB^u$ is closed under product and completes the  existence part of the proof. 

In order to prove the uniqueness, suppose  $\sD\subset\M(\C_0^u(\GG))$ satisfies the required conditions. Then the Podle\'s condition  yields 
\[
\sD =\bigl\{(\omega\otimes\id)((\Lambda_\GG\otimes\id)\Delta_\GG^u(d))\st d\in \sD, \omega\in B(\Ltwo(\GG))_* \bigr\}^{\text{\tiny{norm-cls}}}.
\] 
Using \eqref{LDDrL} we get
\[\begin{split}
\sD &=\bigl\{(\omega\otimes\id)((\Lambda_\GG\otimes\id)\Delta_\GG^u(d))\st d\in \sD, \omega\in B(\Ltwo(\GG))_* \bigr\}^{\text{\tiny{norm-cls}}}\\
&=\bigl\{(\omega\otimes\id)\Delta_r^{r,u}(b)\st b\in \sB, \omega\in B(\Ltwo(\GG))_*\bigr\}^{\text{\tiny{norm-cls}}}=\sB^u
\end{split}\]
and we get uniqueness. 
\end{proof}
\begin{example} Let  $\HH$ be a closed quantum subgroup of a regular locally compact quantum group $\GG$. In   \cite{Vaes} among other things the $\C^*$-quantum homogeneous space   $\C_0(\GG/\HH)\subset \M(\C_0(\GG))$ was constructed (under the regularity condition on $\GG$). Using Theorem \ref{lift} we get its universal version $\C_0^u(\GG/\HH)$. The existence of $\C_0^u(\GG/\HH)$ was hinted  in \cite[Remark 8.6]{Vaes}. 
\end{example}
 \begin{remark}
Let $\sN\subset \Linf(\GG)$ be a von Neumann algebra  satisfying   Baaj-Vaes conditions. Then there exists a locally compact quantum group $\HH$ such that $\sN=\Linf(\HH)\subset \Linf(\GG)$. Using the techniques of the proof of \cite[Proposition 2.3]{KKS} we observe that  $\C_0(\HH)$ embeds into $\M(\C_0(\GG))$. Let us denote the embedding by $\pi: \C_0(\HH)\to \M(\C_0(\GG))$. Then $\pi$  satisfies 
\[\Delta_\HH\circ\pi = (\pi\otimes\pi)\circ\Delta_\GG.\] Using \cite{MRW} we get the universal lift   $\pi^u: \C_0^u(\HH)\to \M(\C_0^u(\GG))$ defining a quantum group homomorphism from $\GG$ to $\HH$. 

Let us note that   $\sB=\pi(\C_0(\HH))$ satisfies Podle\'s condition
\[
\Delta_\GG(\sB)(\C_0(\GG)\otimes\I) = \C_0(\GG)\otimes \sB.\] 
  It is easy to verify that $ \pi^u(\C_0^u(\HH))$ satisfies the conditions of Theorem \ref{lift}, i.e. $\sB^u = \pi^u(\C_0^u(\HH))$.  In particular, $\sB^u$ exists without the  regularity of $\GG$. It may happen that $\pi^u$    is not injective so we don't have $\sB^u \cong\C_0^u(\HH)$ in general. For an example when $\pi^u $ is not injective we refer to  \cite{BV}.
\end{remark}
In the next definition we adopt the terminology introduced in Definition \ref{weakPodandPod} and Definition \ref{weakPodandPoduni}. 
\begin{definition}\label{uniqueness} 
Let   $\sB\subset\M(\C_0(\GG))$ be a non-zero $\C^*$-subalgebra  satisfying    weak Podle\'s condition.
We say that a $\C^*$-subalgebra    $\sB^u\subset\M(\C^u_0(\GG))$ is a {\it weak lift} of   $\sB$ if  $\Lambda_\GG(\sB^u) = \sB$ and   $\sB^u$ satisfies weak Podle\'s condition. 

Let   $\sB\subset\M(\C_0(\GG))$ be a non-zero $\C^*$-subalgebra  satisfying    Podle\'s condition.
We say that a $\C^*$-subalgebra    $\sB^u\subset\M(\C^u_0(\GG))$ is a {\it lift} of   $\sB$ if  $\Lambda_\GG(\sB^u) = \sB$ and   $\sB^u$ satisfies  Podle\'s condition. 
\end{definition}
Using Proposition \ref{streg1} we get 
\begin{corollary}
Let $\sB\subset \M(\C_0(\GG))$ be a non-zero $\C^*$-algebra satisfying weak Podle\'s condition  and admitting  a weak lift  $\sB^u$. Then $\sB^u$ is non-degenerate,  $\sB^u \C_0^u(\GG) = \C_0^u(\GG)$.
\end{corollary}
\begin{remark}
It can be easily proved that if the weak lift  $\sB^u$ of $\sB$  in the sense of Definition \ref{uniqueness} exists   then it is uniquely given by
\begin{equation}\label{defBu}
\bigl\{(\omega\otimes\id)(\Delta_r^{r,u}(b))\st b\in\sB, \omega\in B(\Ltwo(\GG))_*\bigr\}^{\text{\tiny{norm-cls}}}.
\end{equation} Conversely, suppose that  $\sB\subset\M(\C_0(\GG))$ is a non-zero $\C^*$-subalgebra  satisfying  weak Podle\'s condition. Let us consider  $\sB^u$ defined by \eqref{defBu}. Then $\sB^u$ is a $*$-closed linear subspace of $\M(\C_0^u(\GG))$. Using the techniques of the first part of the proof of Theorem \ref{lift} we can prove (not using  regularity) that $\sB^u$ satisfies weak Podle\'s condition.  Note again the abuse of terminology as in the proof of Theorem \ref{lift}.  Assuming semi-regularity of $\GG$ and using the techniques of the proof of  \cite[Proposition 5.7]{BSV} we prove that  $\sB^u$ is a $\C^*$-subalgebra of $\M(\C_0^u(\GG))$.  

We were not able to prove that if $\sB\subset\M(\C_0(\GG))$ satisfies weak Podle\'s condition  then $\Delta_\GG(\sB)\subset \M(\C_0(\GG)\otimes\sB)$. Suppose that   $\Delta_\GG(\sB)\subset\M(\C_0(\GG)\otimes\sB)$ holds and $\sB$ admits a weak lift $\sB^u$.  In this case we were not able to prove that  $\Delta_\GG^u(\sB^u)\subset\M(\C^u_0(\GG)\otimes\sB^u)$.  Note again that the latter  holds   if    $\GG$ is regular. The above discussion yields the following definition. 
\end{remark}

\begin{definition}\label{sul}
Let $\sB\subset\M(\C_0(\GG))$ be a non-zero $\C^*$-subalgebra  satisfying    weak Podle\'s condition and such that $\Delta_\GG(\sB)\subset \M(\C_0(\GG)\otimes \sB)$.  We say that a $\C^*$-subalgebra    $\sB^u\subset\M(\C^u_0(\GG))$ is a \emph{strong  lift} of   $\sB$ if  $\Lambda_\GG(\sB^u) = \sB$, $\sB^u$ satisfies  weak Podle\'s condition  and $\Delta^u_\GG(\sB^u)\subset \M(\C^u_0(\GG)\otimes \sB^u)$. 
\end{definition}

In what follows   we shall consider the behavior of symmetry of $\sB$ under the      lift to the universal level $\sB^u$. 
\begin{proposition}\label{sym}
Let $\sB$ be a non-zero $\C^*$-subalgebra of $\M(\C_0(\GG))$    satisfying   weak Podle\'s condition and such that  $\Delta_\GG(\sB)\subset\M(\C_0(\GG)\otimes\sB)$. If    $\sB$ admits   strong lift  $\sB^u$ then $\sB^u$   is symmetric, i.e. ${\wW^\GG}(\I\otimes \sB^u){\wW^\GG}^*\subset \M(\C_0(\hh\GG)\otimes \sB^u)$. 
\end{proposition}
\begin{proof}
Let us first note that $\Delta_\GG^u(\sB^u)\subset \M(\C^u_0(\GG)\otimes\sB^u)$ implies that $\Delta_r^{r,u}(\sB)\subset \M(\C_0(\GG)\otimes\sB^u)$. 
Thus using ${\ww^\GG}(\I\otimes \sB){\ww^\GG}^* \subset \M(\C_0(\hh\GG)\otimes \sB)$  we get  
\[\ww^\GG_{12}\wW^\GG_{13}(\I\otimes \Delta_r^{r,u}(\sB))\wW^{\GG^*}_{13}\ww^{\GG^*}_{12}\subset \M(\C_0(\hh\GG)\otimes \C_0(\GG)\otimes \sB^u).\] In particular 
$\wW^\GG_{13}(\I\otimes\Delta_r^{r,u}(\sB))\wW^{\GG^*}_{13}\subset \M(\C_0(\hh\GG)\otimes \C_0(\GG)\otimes\sB^u)$ which when sliced with $(\id\otimes\mu\otimes \id)$, where $\mu$ runs over all functionals on $\C_0(\GG)$ yields  ${\wW^\GG}(\I\otimes \sB^u){\wW^\GG}^*\subset \M(\C_0(\hh\GG)\otimes \sB^u)$
\end{proof}
Before formulating the next definition we refer to \eqref{bet} for a  definition of the action  $\beta\in\Mor(\C_0(\GG),\C_0(\hh\GG)\otimes\C_0(\GG))$ of $\hh\GG$ on $\C_0(\GG)$.
Let $\sB\subset\M(\C_0(\GG))$ be a non-degenerate $\C^*$-subalgebra such that $\beta(\sB)\subset \M(\C_0(\hh\GG)\otimes \sB)$. We say that 
  $\sB$ satisfies 
\begin{itemize}
\item \emph{$\beta$-Podle\'s condition} if $\beta(\sB)(\C_0(\hh\GG)\otimes\I)=\C_0(\hh\GG)\otimes\sB$,
\item \emph{$\beta$-weak Podle\'s condition} if $(\C_0(\hh\GG)\otimes\I)\beta(\sB)(\C_0(\hh\GG)\otimes\I)=\C_0(\hh\GG)\otimes\sB$.
\end{itemize}
Similarly we define  $\beta^u:\C^u_0(\GG)\to \M(\C_0(\hh\GG)\otimes\C^u_0(\GG))$ 
 \[\beta^u(x) =  \wW^\GG(\I\otimes x){\wW^\GG}^*\]
  for all $x\in\C^u_0(\GG)$ and the notion of \emph{$\beta^u$-(weak) Podle\'s condition} for a non-degenerate $\sD\subset \M(\C_0^u(\GG))$.
\begin{example}\label{betapodquot}
  Let $\GG$ be a regular locally compact quantum group and $\HH$ a closed quantum subgroup of $\GG$. Using \cite[Theorem 8.2]{Vaes} we see that $\C_0(\GG/\HH)$ satisfies $\beta$-Podle\'s condition.  
  \end{example}
\begin{lemma}\label{betaP}
Let  $\sB\subset \M(\C_0(\GG))$ be a symmetric  $\C^*$-subalgebra  satisfying weak Podle\'s and $\beta$-weak Podle\'s condition. Suppose that $\sB $   admits a weak lift $\sB^u$. Then $\sB^u$ satisfies $\beta^u$-weak Podle\'s condition
\[
(\C_0(\hh\GG)\otimes\I)\beta^u(\sB^u)(\C_0(\hh\GG)\otimes\I)=\C_0(\hh\GG)\otimes\sB^u.
\]
\end{lemma}
\begin{proof}
$\beta$-weak  Podle\'s condition  for $\sB$ writes
\begin{equation}\label{beta}
(\C_0(\hh\GG)\otimes\I)\bigl({\ww^\GG}(\I\otimes\sB){\ww^\GG}^*\bigr)(\C_0(\hh\GG)\otimes\I) = \C_0(\hh\GG)\otimes\sB
\end{equation}
Applying $\id\otimes\Delta_r^{r,u}$ to both sides of \eqref{beta} and then slicing the  second leg by the functionals on $\C_0(\GG)$ we get
\begin{align*}
\C_0(\hh\GG)\otimes\sB^u &= \bigl\{(\id\otimes\mu\otimes\id)\bigl((\tilde y\otimes\I\otimes\I)\ww^\GG_{12}\wW^\GG_{13}(\I\otimes\Delta_r^{r,u}(b))\wW^{\GG^*}_{13}\ww^{\GG^*}_{12}(y\otimes\I\otimes\I)\bigr)\st 
\\&\quad\quad\quad\quad\quad\quad\quad\quad\quad\quad
\mu\in\C_0(\GG)^*, b\in\sB, y,\tilde y\in\C_0(\hh\GG)\bigr\}^{\text{\tiny{norm-cls}}}\\
&=\bigl\{(\id\otimes\mu\otimes\id)\bigl((\tilde y\otimes\I\otimes\I)\wW^\GG_{13}(\I\otimes\Delta_r^{r,u}(b))\wW^{\GG^*}_{13}(y\otimes\I\otimes\I)\bigr)\st 
\\&\quad\quad\quad\quad\quad\quad\quad\quad\quad\quad
\mu\in\C_0(\GG)^*, b\in\sB, y,\tilde y\in\C_0(\hh\GG)\bigr\}^{\text{\tiny{norm-cls}}}\\
&=\bigl\{(\id\otimes\mu\otimes\id)\bigl((\tilde y\otimes\I\otimes\I)\wW^\GG_{13}(\I\otimes(\tilde{x}\otimes\I)\Delta_r^{r,u}(b)(x\otimes\I))\wW^{\GG^*}_{13}(y\otimes\I\otimes\I)\bigr)\st 
\\&\quad\quad\quad\quad\quad\quad\quad\quad\quad\quad
\mu\in\C_0(\GG)^*, b\in\sB, x, \tilde{x}\in\C_0(\GG), y,\tilde y\in\C_0(\hh\GG)\bigr\}^{\text{\tiny{norm-cls}}}\\
&=\bigl\{(\id\otimes\mu\otimes\id)\bigl((\tilde y\otimes\I\otimes\I)\wW^\GG_{13}(\I\otimes x\otimes b)\wW^{\GG^*}_{13}(y\otimes\I\otimes\I)\bigr)\st
\\&\quad\quad\quad\quad\quad\quad\quad\quad\quad\quad
 \mu\in\C_0(\GG)^*, b\in\sB^u, x\in\C_0(\GG), y\in\C_0(\hh\GG)\bigr\}^{\text{\tiny{norm-cls}}}\\
&=\bigl\{(\tilde y \otimes\I)\bigl({\wW^\GG}(\I\otimes b){\wW^\GG}^*\bigr)(y\otimes\I)\st 
y,\tilde y\in\C_0(\hh\GG), b\in\sB^u\bigr\}^{\text{\tiny{norm-cls}}}\\
&=(\C_0(\hh\GG)\otimes\I)\beta^u(\sB^u)(\C_0(\hh\GG)\otimes\I).
\end{align*} 
\end{proof}
\begin{definition}\label{strnsymm}
Let $\sD$ be a non-degenerate $\C^*$-subalgebra of $\M(\C_0^u(\GG))$. We say that $\sD$ is \emph{strongly symmetric} if $\WW^\GG(\I\otimes \sD){\WW^\GG}^*\subset\M(\C^u_0(\hh\GG)\otimes \sD)$.
\end{definition}
In the following lemma we will establish a relation between the concepts of symmetric and strongly symmetric subalgebras of $\M(\C_0^u(\GG))$.  
\begin{lemma}\label{stsymmlift} Let  $\sB\subset \M(\C_0(\GG))$ be a symmetric  subalgebra  satisfying weak Podle\'s condition,  $\beta$-weak Podle\'s condition and such that $\Delta_\GG(\sB)\subset\M(\C_0(\GG)\otimes\sB)$.  Suppose that $\sB $  admits strong lift $\sB^u$.  Then $\sB^u$ is strongly symmetric and
\begin{equation}\label{hiperbeta}
(\C_0^u(\hh\GG)\otimes\I)\WW^\GG (\I\otimes \sB^u){\WW^\GG}^*(\C_0^u(\hh\GG)\otimes\I) = \C_0^u(\hh\GG)\otimes\sB^u
\end{equation}
\end{lemma}
\begin{proof}
$\sB$ is symmetric, so using Proposition \ref{sym} we get
\begin{equation}\label{hsym}
{\wW^\GG}(\I\otimes\sB^u){\wW^\GG}^*\subset \M(\C_0(\hh\GG)\otimes\sB^u).
\end{equation}
Applying $(\hh{\Delta}_r^{r,u}\otimes\id)$ to \eqref{hsym}   we observe that
\begin{equation}\label{gsym}
(\hh{\Delta}_r^{r,u}\otimes\id){\wW^\GG}(\I\otimes\sB^u){\wW^\GG}^*\subset \M(\C_0(\hh\GG)\otimes\C_0^u(\hh\GG)\otimes\sB^u). 
\end{equation}
By slicing the first leg of \eqref{gsym} we obtain a subset of $\M(\C_0^u(\hh\GG)\otimes\sB^u)$ and now we compute the left hand side
\begin{align*}
\M(\C_0^u(\hh\GG)\otimes\sB^u)&\supset\bigl\{ (\mu\otimes\id\otimes\id)(\hh{\Delta}_r^{r,u}\otimes\id)\wW^\GG(\I\otimes b)\wW^{\GG^*}\st \mu\in\C_0(\hh\GG)^*, b\in\sB^u \bigr\}^{\text{\tiny{norm-cls}}}\\
&=\bigl\{ (\mu\otimes\id\otimes\id)\WW^\GG_{23}\wW^\GG_{13}(\I\otimes\I\otimes b)\wW^{\GG^*}_{13}\WW^{\GG^*}_{23}\st \mu\in\C_0(\hh\GG)^*, b\in\sB^u \bigr\}^{\text{\tiny{norm-cls}}}\\
&=\bigl\{ (\mu\otimes\id\otimes\id)\WW^\GG_{23}\bigl((\tilde{y}\otimes\I\otimes\I)\wW^\GG_{13}(\I\otimes\I\otimes b)\wW^{\GG^*}_{13}(y\otimes\I\otimes\I)\bigr)\WW^{\GG^*}_{23}\st
\\&\quad\quad\quad\quad\quad\quad\quad\quad\quad\quad
 \mu\in\C_0(\hh\GG)^*, \tilde{y},y\in\C_0(\hh\GG), b\in\sB^u \bigr\}^{\text{\tiny{norm-cls}}}\\
&=\bigl\{ (\mu\otimes\id\otimes\id)\bigl(\WW^\GG_{23}(x\otimes\I\otimes b)\WW^{\GG^*}_{23}\bigr)\st  \mu\in\C_0(\hh\GG)^*, x\in\C_0(\hh\GG), b\in\sB^u \bigr\}^{\text{\tiny{norm-cls}}}\\
&=\WW^\GG (\I\otimes \sB^u){\WW^\GG}^* 
\end{align*}
where in the third equality we use  Lemma \ref{betaP}.
In order to prove \eqref{hiperbeta} we compute 
\begin{align*}
\C^u_0(\hh\GG)\otimes\sB^u &= \bigl\{(\mu\otimes\id\otimes\id)\big(\hh\Delta_r^{r,u}(\C_0(\hh\GG))\otimes\sB^u\big)\st 
\mu\in\C_0(\hh\GG)^* \bigr\}^{\text{\tiny{norm-cls}}}\\ 
&=\bigl\{ (\mu\otimes\id\otimes\id)((\hh\Delta_r^{r,u}(y)\otimes\I)\WW^\GG_{23}\wW^\GG_{13} (\I\otimes\I\otimes b)\wW^{\GG^*}_{13} \WW^{\GG^*}_{23}(\hh\Delta_r^{r,u}(\tilde{y})\otimes\I))\st
\\&\quad\quad\quad\quad\quad\quad\quad\quad\quad\quad
 \mu\in\C_0(\hh\GG)^*, y,\tilde{y}\in\C_0(\hh\GG), b\in\sB^u \bigr\}^{\text{\tiny{norm-cls}}} \\ 
&=\bigl\{ (\mu\otimes\id\otimes\id)(( y\otimes z\otimes \I)\WW^\GG_{23}\wW^\GG_{13} (\I\otimes\I\otimes b)\wW^{\GG^*}_{13} \WW^{\GG^*}_{23}( \tilde{y} \otimes\tilde{z}\otimes \I))\st
\\&\quad\quad\quad\quad\quad\quad\quad\quad\quad\quad
 \mu\in\C_0(\hh\GG)^*, y, \tilde{y} \in\C_0(\hh\GG), z,\tilde{z}\in\C^u_0(\hh\GG),  b\in\sB^u \bigr\}^{\text{\tiny{norm-cls}}} \\ 
&=\bigl\{ (\mu\otimes\id\otimes\id)(( y\otimes z\otimes \I)\WW^\GG_{23}  (\I\otimes\I\otimes b)  \WW^{\GG^*}_{23}( \tilde{y} \otimes\tilde{z}\otimes \I))\st
\\&\quad\quad\quad\quad\quad\quad\quad\quad\quad\quad
 \mu\in\C_0(\hh\GG)^*, y, \tilde{y} \in\C_0(\hh\GG), z,\tilde{z}\in\C^u_0(\hh\GG),  b\in\sB^u \bigr\}^{\text{\tiny{norm-cls}}} \\&=(\C^u_0(\hh\GG)\otimes\I)\WW^\GG   (\I \otimes \sB^u)  {\WW^\GG}^* (\C^u_0(\hh\GG)\otimes\I)
\end{align*}
 where in the fourth equality we use  Lemma \ref{betaP}. Thus we get Equation \eqref{hiperbeta}. 
\end{proof}
Using Lemma \ref{stsymmlift}, Lemma \ref{weakPodreg} and Theorem \ref{lift} in the context of $\C_0(\GG/\HH)$ we get
\begin{proposition}
Let $\GG$ be a regular locally compact quantum group and $\HH$   a closed quantum subgroup of $\GG$. Then $\C_0^u(\GG/\HH)$  is strongly symmetric. Moreover  it satisfies Podle\'s condition 
\[\Delta_\GG^u(\C_0^u(\GG/\HH))(\C_0^u(\GG)\otimes\I) = \C_0^u(\GG)\otimes\C_0^u(\GG/\HH)\] and the following holds 
\[(\C_0^u(\hh\GG)\otimes\I)\WW^\GG(\I\otimes\C_0^u(\GG/\HH)){\WW^\GG}^* = \C_0^u(\hh\GG)\otimes\C_0^u(\GG/\HH).\] 
\end{proposition}
 
Let $\HH\subset \GG$ be  a compact quantum group. In the final part of this section we shall analyze the covariance  of the conditional expectation $E:\Linf(\GG)\to\Linf(\GG/\HH)$ under $\beta:\Linf(\GG) \to\Linf(\hh\GG)\vtens\Linf( \GG)$. For the formula for $E$ we refer  to \eqref{eq:def_ce_H}. 
\begin{definition} \label{betacov}
Let $\GG$ be a locally compact quantum group,  $\HH$ a compact quantum subgroup of $\GG$,  $E:\Linf(\GG)\to\Linf(\GG)$ the conditional expectation onto $\Linf(\GG/\HH)$ and $\beta:\Linf(\GG)\to\Linf(\hh\GG)\vtens\Linf(\GG)$ the action of $\hh\GG$ on $\Linf(\GG)$ defined in \eqref{betvN}.  We say that $E$ is \emph{$\beta$-covariant} if
\[(\id\otimes E)\circ \beta = \beta\circ E.\]
\end{definition}
Let $\HH$ be a compact quantum group and let us consider the map   $\Theta:\Linf(\HH)\to \Linf(\hh\HH)$  
\[\Theta(x) = (\id\otimes \psi_\HH)({\ww^\HH}^*(\I\otimes x){\ww^\HH}).\]
In \cite[Corollary 3.9]{Izumi} it was proved that $\Theta(x) = \psi_\HH(x)\I$ if and only if $\HH$ is of Kac type. Using the  invariance of the Haar measure under the unitary antipode $\psi_\HH= \psi_\HH\circ R_\HH$ and the formula $(R_{\hh\HH}\otimes R_\HH)(\ww^\HH) = \ww^\HH$ we get the following result. 
\begin{corollary}\label{coro}
Let $\HH$ be a  compact quantum group. Then 
\[(\id\otimes \psi_\HH)({\ww^\HH} (\I\otimes x){\ww^\HH}^*) = \psi_\HH(x)\I\] for all $x\in \Linf(\HH)$ if and only if $\HH$ is of Kac type. 
\end{corollary}
\begin{proposition}\label{lemtr}
Let $\HH$,   $\GG$, $\beta$ and $E$ be as in Definition \ref{betacov}. Then  $E$  is $\beta$-covariant if and only if $\HH$ is of Kac type. 
\end{proposition}
\begin{proof} 
Let $\alpha: \Linf(\GG)\to  \Linf(\GG)\vtens\Linf(\HH) $ be the corresponding action of $\HH$ on $\Linf(\GG)$. Using  \cite[Theorem 3.6 (3)]{DKSS} we conclude that 
\begin{equation}\label{linfH}
\Linf(\HH) = \big\{(\omega\otimes\id)(\alpha(x)):\omega\in\Linf(\GG)_*, x\in\Linf(\GG)\big\}^{\textrm{\tiny{$\sigma$-weak cls}}}.
\end{equation} 
 The conditional expectation is given by the formula $E=(\id\otimes \psi_\HH)\circ\alpha$ thus using the formula  $(\id\otimes\alpha)(\ww^\GG) = \ww^\GG_{12}V_{13}$, where $V$ is given by \eqref{eq:def_V},  we get
\begin{align*}(\id\otimes E)\circ\beta(x)&=
(\id\otimes E)({\ww^\GG}(\I\otimes x){\ww^\GG}^*) \\& = {\ww^\GG} (\id\otimes \id\otimes \psi_\HH)(V_{13}(\I\otimes \alpha(x))V_{13}^*){\ww^\GG}^*.
\end{align*}
On the other hand 
\[\beta\circ E(x) = \ww^\GG(\I\otimes(\id\otimes\psi_\HH)\circ\alpha(x)){\ww^\GG}^*.\]
In particular $E$ is $\beta$-covariant if  and only if 
\[(\id\otimes \id\otimes \psi_\HH)(\ww^\HH_{13}(\I\otimes \alpha(x))\ww^{\HH^*}_{13}) = \I\otimes\big((\id\otimes\psi_\HH)\circ\alpha(x)\big) \] 
 which by \eqref{linfH} is equivalent with 
\[(\id\otimes \psi_\HH)(\ww^\HH (\I\otimes x){\ww^\HH}^* ) =\psi_\HH (x) \I \] for all $x\in\Linf(\HH)$. 
Using Corollary \ref{coro} we get the desired equivalence. 
\end{proof}
Let $\GG$ be a  regular locally compact quantum group and  
  $\HH$  a Kac type compact quantum subgroup of  $\GG$. In this case the proof that $\C_0(\GG/\HH)$ satisfies   $\beta$-Podle\'s is easy (see Example \ref{betapodquot} for the discussion of the general case) and follows from the following computation. 
\begin{align*}\beta(\C_0(\GG/\HH))(\C_0(\hh\GG)\otimes \I)&= \beta(E(\C_0(\GG)))(\C_0(\hh\GG)\otimes \I)\\
&=(\id\otimes E)(\beta(\C_0(\GG))(\C_0(\hh\GG)\otimes \I))\\
&=(\id\otimes E)(\C_0(\hh\GG)\otimes \C_0(\GG)) = \C_0(\hh\GG)\otimes \C_0(\GG/\HH).
\end{align*}

 \section{Integrable coideals and idempotent states on $\GG$}\label{integcoid}
Let $\omega\in S(\C_0^u(\GG))$ be an idempotent state on $\GG$ and let $E:\Linf(\GG)\to\Linf(\GG)$ be the conditional expectation assigned to $\omega$. Let $\sN = E( \Linf(\GG))$ be  a coideal in $\Linf(\GG)$ assigned to $\omega$. 
\begin{proposition}\label{proptauin}
Let $\omega$ and $\sN$ be as above. Then 
\begin{itemize}
\item $\sN$ is integrable;
\item $\tau_t(\sN) = \sN$ for all $t\in\mathbb{R}$.
\end{itemize}
\end{proposition}
\begin{proof}Integrability of $\sN$ is equivalent with $\psi_\GG|_\sN$ being semifinite. 
Let us take $x\in\Linf(\GG)$, $x\geq 0$. Then we have 
\[\begin{split}\psi_\GG(x)\I &= (\psi_\GG\otimes\id)((\id\otimes E)\circ\Delta_\GG(x))\\
&=(\psi_\GG\otimes\id)(\Delta_\GG(E(x))\\
&=\psi_\GG(E(x))\I
\end{split}\] which implies the required semifiniteness. Using \eqref{prescgr1} we get \[\tau_t(\sN) = \tau_t(E(\Linf(\GG))) = E( \tau_t(\Linf(\GG))) = \sN.\]
\end{proof}
In what follows we shall prove the converse of Proposition \ref{proptauin}. 
\begin{theorem}\label{thminterg}
Let $\sN$ be an integrable coideal of $\Linf(\GG)$ which is $\tau_t$-invariant. Then there exists  
a unique conditional expectation $E: \Linf(\GG)\to\Linf(\GG)$ onto $\sN$ such that $\psi_\GG(x) = \psi_\GG(E(x))$ for all $x\in\Linf(\GG)^+$. Moreover   there exists a unique idempotent state $\omega\in \C_0^u(\GG)^*$ such that for every $x\in \Linf(\GG)$, $E(x)=\omega\staru x$
and $\sN=\bigl\{ x\in\Linf(\GG) : \omega\staru x = x \bigr\}$.
\end{theorem}
\begin{proof}
$\sN$ is an integrable von Neumann subalgebra of $\Linf(\GG)$ so the restriction of right Haar weight $\psi_\GG$ to $\sN$ is semifinite.  Let us show that since $\sN$ is preserved by $\tau_t$,  it must be also preserved by  $\sigma^{\prime}_t$. For the latter we first use \eqref{modular} and the fact that $\sN$ is a coideal, i.e.
\begin{equation}\label{sigma-inv}
(\sigma_t\otimes\sigma^{\prime}_{-t})\circ\Delta_\GG(\sN) = \Delta_\GG\circ\tau_t(\sN) = \Delta_\GG  (\sN)\subset \Linf(\GG)\vtens\sN.
\end{equation} Slicing \eqref{sigma-inv} and using \cite[Corollary 2.6]{embed} we get 
\[\begin{split}\sN &= \sigma^{\prime}_{-t} \left(\big\{(\omega\circ \sigma_t\otimes\id)(\Delta_\GG(\sN))\st \omega\in B(\Ltwo(\GG))_*\big\}^{\textrm{\tiny{$\sigma$-weak cls}}}\right)\\&= \sigma^{\prime}_{-t} \left(\big\{(\omega\otimes\id)(\Delta_\GG(\sN))\st\omega\in B(\Ltwo(\GG))_*\big\}^{\textrm{\tiny{$\sigma$-weak cls}}}\right) = \sigma^{\prime}_{-t}(\sN).
\end{split}\]
 Using \cite[Theorem IX.4.2]{Tak} we conclude that  there exists a unique normal conditional expectation $E:\Linf(\GG)\to \Linf(\GG)$ onto $\sN$ which preserves $\psi_\GG$. Using the Kadison inequality \[E(x)^*E(x)\leq E(x^*x)\] we can see that $E$ admits the Hilbert space extension, i.e. a projection  $P:\Ltwo(\GG)\to\Ltwo(\GG)$ such that 
 \begin{equation}\label{def:P}P\eta_\GG(x) = \eta_\GG(E(x))\end{equation} for all $x\in\Linf(\GG)$ such that $\psi_\GG(x^*x)<\infty$.  

Let us note that since $E$ is a conditional expectation onto $\sN$, the coideal property of $\sN$ yields
\begin{equation}\label{eqdef}
 (\id\otimes E)\circ\Delta_\GG\circ E=\Delta_\GG\circ E
\end{equation}   
Our aim now is to show that  $(\id\otimes E)\circ\Delta_\GG=\Delta_\GG\circ E$. In order to prove it
let us consider the canonical implementation  $U\in\Linf(\GG)\bar\otimes\B(\Ltwo(\GG))$ of $\Delta_\GG$ (see \cite{VaesCan}). Let us  recall that $U$ satisfies 
\[\begin{split}\Delta_\GG(x) &= U(\I\otimes x)U^*,\\
(\Delta_\GG\otimes\id)(U) &= U_{23}U_{13}.
\end{split}
\] 
Using \cite[Proposition 2.4]{VaesCan} we get the explicit formula defining $U$:
 \begin{equation}\label{def:U_ptop}\big((\omega_{\xi,\delta^{\frac{1}{2}}\zeta}\otimes\id)U\big)\eta_\GG(x) = \eta_\GG\big((\omega_{\xi, \zeta}\otimes\id)\Delta_\GG(x)\big)\end{equation}    for all $\zeta\in D(\delta^{\frac{1}{2}})$; here  we view the modular element $\delta$ as an unbounded operator acting on $\Ltwo(\GG)$. 
Slicing \eqref{eqdef} with $(\omega_{\xi, \zeta}\otimes\id)$,   applying $\eta_\GG$ to the result, using \eqref{def:P} and \eqref{def:U_ptop}  we get  $PxP = xP$ for all $x\in \sM$ where  
 \begin{equation}\label{defM} \sM = \big\{(\omega_{\xi, \eta}\otimes\id)(U):\xi,\eta\in\Ltwo(\GG)\big\}^{\textrm{\tiny{$\sigma$-weak cls}}}.\end{equation}
 Since $U$ is a representation of $\GG$,  $\sM$ forms a von Neumann algebra. 
In particular having $PxP = xP$ satisfied   for all $x\in\sM$ we easily conclude that $Px = xP$. This in turn is equivalent with $(\id\otimes E)\circ\Delta_\GG = \Delta_\GG\circ E$. 

We conclude by using  \cite[Theorem 5]{SaS}, which yields the    unique  idempotent state $\omega\in S(\C_0^u(\GG))$ such that 
$E(x)=\omega\staru x$. 
\end{proof}
Having Proposition \ref{proptauin} and Theorem \ref{thminterg} and using \cite[Theorem 1]{SaS} we get 
\begin{corollary}
There is a 1-1 correspondence between integrable   coideals $\sN\subset\Linf(\GG)$ preserved by  $\tau_t$   and idempotent states on $\GG$, where  denoting the conditional expectation given by $\omega$ with  $E:\Linf(\GG)\to\Linf(\GG)$,  we have $\sN = E(\Linf(\GG))$. Moreover   $E$   preserves $\psi_\GG$ and $\varphi_\GG$.
\end{corollary}\begin{remark}\label{remphiexp}
 Let us note that in the course of the proof of Theorem \ref{thminterg} we could not use the left version of \cite[Corollary 3]{SaS} and immediately conclude the existence of $\omega$. This would be possible knowing that   $\varphi_\GG$  is $E$-invariant. Actually using the techniques of the proof of Theorem \ref{thminterg} we can prove   Theorem  \ref{thminterg1}  which is a strengthened  version of \cite[Theorem 1, v]{SaS}: there is a   1-1 correspondence between  idempotent states on $\GG$ and $\psi_\GG$-expected left-invariant von Neumann subalgebras of $\Linf(\GG)$.  
\end{remark} 
\begin{theorem}\label{thminterg1} 
Let $\sN$ be a $\psi_\GG$-expected  coideal  in $\Linf(\GG)$. Then the conditional expectation  $E$ onto $\sN$ satisfies $(\id\otimes E)\circ\Delta_\GG = \Delta_\GG\circ E$. In particular  there exists a unique idempotent state $\omega\in \C_0^u(\GG)^*$ such that for every $x\in \Linf(\GG)$, $E(x)=\omega\staru x$
and $\sN=\bigl\{ x\in\Linf(\GG) : \omega\staru x = x \bigr\}$.
\end{theorem}
\begin{proof}
Since $\Delta_\GG(\sN)\subset\Linf(\GG)\vtens\sN$ we have $(\id\otimes E)\circ\Delta_\GG\circ E = \Delta_\GG\circ E$. Let $P$ be the $\Ltwo(\GG)$-version of $E$ (it exists since $E$ preserves $\psi_\GG$). Proceeding as in the second part of the proof of  Theorem \ref{thminterg} we conclude that $(\id\otimes E)\circ\Delta_\GG = \Delta_\GG\circ E$ and we are done.  
\end{proof}
Let us formulate a result which may be viewed as strengthen version of    \cite[Corollary 3]{SaS}. 
\begin{theorem}
Let $\sB\subset \C_0(\GG)$ be a non-zero left-invariant $\C^*$-subalgebra. Then the following are equivalent:
\begin{itemize}
\item[(i)] $\sB$ is expected;
\item[(ii)] $\sB$ is $\varphi_\GG$-expected;
\item[(iii)] $\sB$ is $\psi_\GG$-expected.
\end{itemize}
\end{theorem}
\begin{proof}
 Using Proposition \ref{streg} we get that  $\sB$ is non-degenerate.  
The equivalence of (i) and (ii) is the left-side counterpart of \cite[Corollary 3]{SaS}. Clearly (i)$\implies$(iii). Suppose that $E:\C_0(\GG)\to\C_0(\GG)$ is a conditional expectation onto a $\C^*$-subalgebra $\sB$ preserving $\psi_\GG$. Let $P:\Ltwo(\GG)\to\Ltwo(\GG)$ be the Hilbert space extension of $E$ (it exists since $E$ preserves  $\psi_\GG$). Using the techniques of the proof of Theorem \ref{thminterg}  we can show that $Px=xP$ for all $x\in\sM$ where $\sM$ was defined in \eqref{defM}. In particular $\Delta_\GG\circ E = (\id\otimes E)\circ\Delta_\GG$.  Using \cite[Lemma 11]{SaS} we extend $E$ to a unital normal  conditional expectation $E:\Linf(\GG)\to \Linf(\GG)$  satisfying $\Delta_\GG\circ E = (\id\otimes E)\circ\Delta_\GG$.  Using  \cite[Theorem 5]{SaS} we get a unique  idempotent state $\omega\in\C_0^u(\GG)^*$ such that 
$E(x)=\omega\staru x$. Finally using \cite[Theorem 1]{SaS} we conclude that $E$ preserves $\varphi_\GG$, i.e. $\sB$ is expected. Thus (iii)$\implies$(i) and we are done.  
\end{proof}
\section{Normal coideals and  compact quantum subgroups} \label{salres} 
In  \cite{Sal}, among other results, the characterization of left-invariant $\C^*$-subalgebras  of $\C_0(\GG)$ corresponding to compact quantum subgroups of $\GG$ were provided    under coamenability assumption on $\GG$. In this section we give a characterization  of left-invariant $\C^*$-subalgebras of  $\C_0^u(\GG)$ corresponding to compact quantum subgroups with  coamenability assumption  dropped. Moreover,  we characterize   coideals of $\Linf(\GG)$   corresponding to compact  quantum subgroups of $\GG$. 
\subsection{Universal $\C^*$-version}
In this subsection we show that there is a 1-1 correspondence between compact quantum subgroups of a locally 
compact quantum groups and  left-invariant, symmetric $\C^*$-subalgebras of $\C_0^u(\GG)$ equipped with a conditional expectation. In order to do it we prove a number of results  in the context of  $\C_0^u(\GG)$ proved in \cite{Sal} under the assumption that $\C_0(\GG) = \C_0^u(\GG)$.  We do not give the proofs  when they are essentially  the same as those given in  \cite{Sal} (the main  difference in the proof then  is that we use $\wW^\GG$ in place of $\ww^\GG$).  We present the proofs  in case we were able to find simplifications.  
We will always assume that a  left-invariant subalgebra   $\sX\subset \C_0^u(\GG)$ satisfies  $\Lambda_\GG(\sX)\neq 0$ (see Proposition \ref{streg1} ). In particular  $\sX  \C_0^u(\GG)=\C_0^u(\GG)$
\begin{definition}
For a $\C^*$-algebra $\sA$ and left-invariant $\C^*$-subalgebra  $\sX\subset\C_0^u(\GG)$,  a non-degenerate $*$-homomorphism $\rho: \C_0^u(\GG)\to \M(\sA)$ is called $\sX$-trivial if for every $x\in\sX$, 
$$\rho(x)=\varepsilon(x) 1_{\M(\sA)}.$$
Denoting the set of all equivalence classes of non-degenerate $\sX$-trivial representations of $\C_0^u(\GG)$ by $T_\sX$  we define an ideal  $I_\sX= \bigcap_{\rho\in T_\sX} \ker\rho$. 
\end{definition}
The proof of the next theorem is essentially the same as the proof of \cite[Theorem 2]{Sal}.
\begin{theorem}
Let $\GG$ be a locally compact quantum group and $\sX\subset \C_0^u(\GG)$ a   left-invariant $\C^*$-subalgebra. There exists a compact quantum subgroup $\HH\subset \GG$  such that $\C^u(\HH) = \C_0^u(\GG)/I_\sX$.
\end{theorem}
\begin{notation}
Let $\sX\subset \C_0^u(\GG)$ be a left-invariant $\C^*$-subalgebra. Then the compact quantum subgroup of $\GG$ assigned to X will be denoted by $\HH_\sX$.
\end{notation}
\begin{definition}
Let $\HH$ be a compact quantum subgroup of $\GG$ and $\pi:\C_0^u(\GG)\to\C^u(\HH)$   the associated
homomorphism.  Define $F=(\ker \pi)^\bot\cap S(\C_0^u(\GG))$. Then we will consider the following sets 
\[\begin{split}
\sX_\HH^1 &= \bigl\{x\in \C_0^u(\GG)\st(\id\otimes\mu)(\Delta_\GG^u(x))=x ~~~\text{for every}~ \mu\in F \bigr\},\\
\sX_\HH^2 &= \bigl\{x\in\C_0^u(\GG)\st (\id\otimes\theta)(\Delta_\GG^u(x))=x ~~~ \text{where }~\theta=\varphi_\HH \pi \bigr\}, \\
\sX_\HH^3 &= \bigl\{x\in\C_0^u(\GG)\st (\id\otimes\pi)(\Delta_\GG^u(x))=x\otimes 1_\HH \bigr\} .
\end{split}\]
\end{definition}
\begin{lemma}\label{eqx}Adopting the above notations we have 
$\sX_\HH^1=\sX_\HH^2=\sX_\HH^3$
\end{lemma}
\begin{proof}
It is trivial that $\sX_\HH^1\subset \sX_\HH^2$. In order to prove that  $\sX_\HH^3\subset\sX_\HH^1$ let us fix  $\mu\in F$, then  
  $\mu=\mu^\prime\circ \pi$ where $\mu^\prime$  is a state of $\C^u(\HH)$. Thus if $x\in \sX_\HH^3$, then 
  \[(\id\otimes \mu)(\Delta_\GG^u(x)) = x\mu^\prime(\I_\HH) = x\] and  $x\in \sX_\HH^1$.  

 Now take $x\in \sX_\HH^2$, i.e. $(\id\otimes\theta)(\Delta_\GG^u(x))=x$. Then 
\[\begin{split}
(\id\otimes\pi)(\Delta_\GG^u(x))&= (\id\otimes\pi) (\Delta_\GG^u((\id\otimes\theta)\Delta_\GG^u(x)))\\&=
(\id\otimes\id\otimes \varphi_\HH )((\id\otimes\pi\otimes\pi)(\Delta_\GG^u\otimes\id)\Delta_\GG^u(x))\\&=
(\id\otimes\id\otimes \varphi_\HH )((\id\otimes\pi\otimes\pi)(\id\otimes\Delta_\GG^u)\Delta_\GG^u(x))\\&=
(\id\otimes\id\otimes \varphi_\HH )((\id\otimes\Delta^u_\HH) (\id\otimes\pi) \Delta_\GG^u(x))\\&=
(\id\otimes \varphi_\HH )((\id\otimes\pi) \Delta_\GG^u(x))\otimes\I = x\otimes\I
\end{split}\]
which shows that $x$ is an element of $\sX_\HH^3$. Summarizing we proved that $\sX_\HH^3\subset \sX_\HH^1\subset \sX_\HH^2\subset \sX_\HH^3$ thus we have the required equalities. 
\end{proof}
From now on we will denote these sets by $\sX_\HH$ and freely use Lemma \ref{eqx}. Note that $\sX_\HH = \C_0^u(\GG/\HH)$. In particular  $\sX_\HH\subset  \C^u_0(\GG)$ is a non-degenerate  left-invariant $\C^*$-subalgebra. 
%Let $\varphi_\HH$ be the Haar state of $\HH$. Put $\theta=\varphi_\HH \pi$, and note that $\theta$ is in $F$. 
\begin{lemma}\label{lemcont} Let $\sX\subset\C_0^u(\GG)$ be  a left-invariant $\C^*$-subalgebra. Then $\sX\subset \sX_{\HH_{\sX}}$. 
\end{lemma}
\begin{proof}
Since $\pi$ is $\sX$-trivial  we have $\pi((\mu\otimes\id)(\Delta_\GG^u(x))) = \mu(x)\I$ for all $\mu\in\C_0^u(\GG)^*$ and $x\in\sX$. Thus 
\[(\id\otimes\pi)(\Delta_\GG^u(x)) =    x\otimes\I\] and we see that  $x\in \sX_{\HH_X}$. 
\end{proof}
 \begin{theorem}
Let $\HH$ be a compact quantum subgroup of a locally compact quantum group $\GG$ and $\sX_\HH\subset \C^u_0(\GG)$ a left-invariant $\C^*$-subalgebra assigned to $\HH$. Then $\sX_\HH$ is   symmetric   and the map $E$ is a conditional expectation from $\C_0^u(\GG)$ onto $\sX_\HH$ such that $(\id\otimes E)\circ\Delta_\GG^u=\Delta_\GG^u\circ E$.
\end{theorem}
\begin{proof} The proof is almost the same as \cite[Theorem 10]{Sal} to be aware of differences we only prove that $\sX_\HH$ is symmetric. We compute 
\[\begin{split}
(\id\otimes\id\otimes\pi)\bigl((\id\otimes\Delta_\GG^u)({\wW^\GG}(\I\otimes x){\wW^\GG}^*)\bigr)&=(\id\otimes\id\otimes\pi)\bigl(\wW^\GG_{12}\wW^\GG_{13}
(1\otimes\Delta_\GG^u(x))\wW^{\GG^*}_{13}\wW^{\GG^*}_{12}\bigr)\\
&=\wW^{\GG}_{12}\vV_{13}(\I\otimes x \otimes \I)\vV_{13}^*\wW^{\GG^*}_{12}\\
&=\wW^\GG_{12}(\I\otimes x \otimes \I)\wW^{\GG^*}_{12}\\
&=\bigl({\wW^\GG}(\I\otimes x){\wW^\GG}^*\bigr) \otimes \I
\end{split}\]
where we define $\vV=(\id\otimes\pi)\wW^\GG$. Thus $(\id\otimes E)({\wW^\GG}\bigl(\I\otimes x){\wW^\GG}^*\bigr)= {\wW^\GG}(\I\otimes x){\wW^\GG}^*$ and we conclude by using   \eqref{eq:ce_mult}. 
\end{proof}
\begin{definition}
Let $\sX$ be a left-invariant subalgebra of $\C_0^u(\GG)$. Then we define
\[
F_0 =\big\{\mu\in S(\C_0^u(\GG)): (\id\otimes\mu)(\Delta^u_\GG(x)) = x \textrm{ for every } x\in\sX\big\}
\] 
\end{definition}
The proof of the next lemma is essentially the same as \cite[Lemma 3]{Sal}.
  \begin{lemma}\label{pre}
\begin{enumerate}
\item $F_0 = \bigl\{ \mu\in S(\C_0^u(\GG))\st \mu=\varepsilon~~ \text{on} ~\sX \bigr\}$.
\item If $\mu\in F_0$, then $\mu(ax) = \mu(a)\mu(x)$ and $\mu(xa) = \mu(x)\mu(a)$ for every $a\in \C_0^u(\GG)$ and $x\in \sX$.
\end{enumerate}
\end{lemma}
The second claim  of Lemma \ref{pre} has the following  obvious extension.  
\begin{lemma}\label{pre1}
Let $H$ be a Hilbert space,  $T\in\M(\mathcal{K}(H)\otimes\sX)$ and $A\in\M(\mathcal{K}(H)\otimes\C^u_0(\GG))$. Then $(\id\otimes\mu)(TA) = (\id\otimes\mu)(T )(\id\otimes\mu)( A)$ for all $\mu\in F_0$. Similarly $(\id\otimes\mu)(AT) = (\id\otimes\mu)(A )(\id\otimes\mu)( T)$.
\end{lemma}
The next lemma  is the universal counterpart of \cite[Lemma 4]{Sal}. We give here a simple proof.  
\begin{lemma}\label{fzero} 
Suppose that $\sX\subset\C_0^u(\GG)$ is a left-invariant and symmetric  $\C^*$-subalgebra. For every $\mu\in F_0$  and $a\in\C_0^u(\GG)$ such that $\mu(a^*a)\neq 0$, the functional $\mu_a(b):={\mu(a^*ba)}/{\mu(a^*a)}$ is in $F_0$.
\end{lemma}
\begin{proof}
It suffices to prove that $\mu(axc) = \mu(x)\mu(ac)$ for all $x\in\sX$ and $a,c\in\C^u_0(\GG)$. Actually it suffices to prove the latter for $a = (\omega\otimes\id)(\wW^\GG)$. We compute 
\begin{align*}
\mu(axc)&= \omega((\id\otimes\mu)(\wW^\GG(\I\otimes xc))\\
& =\omega((\id\otimes\mu)(\wW^\GG(\I\otimes x)\wW^{\GG^*}\wW^\GG(\I\otimes c))\\
& = \omega((\id\otimes\mu)(\wW^\GG(\I\otimes x)\wW^{\GG^*})(\id\otimes\mu)(\wW^\GG(\I\otimes c)))\\
& = \omega((\id\otimes\varepsilon)(\wW^\GG(\I\otimes x)\wW^{\GG^*})(\id\otimes\mu)(\wW^\GG(\I\otimes c)))\\
& = \varepsilon(x)\omega( (\id\otimes\mu)(\wW^\GG(\I\otimes c)))\\
& = \mu(x) \mu(ac)
\end{align*}
where in the third equality we use symmetry  of $\sX$ and Lemma \ref{pre1} with $T = \wW^\GG(\I\otimes x){\wW^\GG}^*$ and $A = \wW^\GG(\I\otimes b)$; in the fifth equality we use $(\id\otimes\varepsilon)(\wW^\GG) =\I$. 
\end{proof}
Using Lemma \ref{fzero} we can prove the universal counterpart of \cite[Theorem 5]{Sal}
\begin{theorem}\label{mains4}
Let $\GG$ be a locally compact quantum group and $\sX\subset\C_0^u(\GG)$  a left-invariant  symmetric $\C^*$-subalgebra. A state $\mu\in S(\C_0^u(\GG))$ is in $F_0$ if and only if its GNS-representation is $\sX$-trivial. Moreover, if $(\HH,\pi)$ is the compact subgroup associated to $\sX$, then $F_0=\pi^*(S(\C^u(\HH)))$ where $\pi^*:\C^u(\HH)^*\to\C_0^u(\GG)^*$ is the adjoint of the quotient map $\pi$. 
\end{theorem}
Finally let us prove the universal counterpart of \cite[Theorem 10]{Sal}.
\begin{theorem}\label{mains5}
Let $\GG$ be a locally compact quantum group and $\sX$   a non-zero, symmetric, left-invariant $\C^*$-subalgebra of $\C_0^u(\GG)$ such that there is a conditional expectation 
$E : \C_0^u(\GG)\to\C_0^u(\GG)$ onto $\sX$ satisfying  $(\id\otimes E)\circ\Delta_\GG^u = \Delta_\GG^u \circ E$ . Then $\sX_{\HH_\sX}=\sX$.
\end{theorem}
\begin{proof}
Using Lemma \ref{lemcont} we get  $\sX\subset \sX_{\HH_\sX}$. Conversely, let $a\in \sX_{\HH_\sX}$. Clearly  $\varepsilon\circ E\in F_0$ and  since in our case $F_0 = F$ (see Theorem \ref{mains4})   we have
$$a=(\id\otimes \varepsilon\circ E )(\Delta_\GG^u(a))=(\id\otimes \varepsilon )(\Delta_\GG^u (E(a)))=E(a).$$
So $a$ is in the image of $E$ which is $\sX$.
\end{proof}
Let $\HH$ be a compact quantum subgroup of $\GG$. Then $\Linf(\hh\HH)$ is a codual coideal of  $\Linf(\GG/\HH)$. Since $\Linf(\GG/\HH) = \Lambda_\GG(\sX_\HH)''$, $\sX_\HH$ can be used to recover $\HH$, i.e.  the assignment $\HH\mapsto\sX_{\HH}$ is injective. Theorem \ref{mains5} yields $\sX_{\HH}=\sX_{\HH_{\sX_{\HH}}}$ and we conclude that $\HH_{\sX_{\HH}} = \HH$. Summarizing we get
 \begin{theorem}\label{1to1coresp}
Let $\GG$ be a locally compact quantum group. There is a 1-1 correspondence between compact quantum subgroups of $\GG$ and  symmetric, left-invariant $\C^*$-subalgebras $\sX$ of $\C_0^u(\GG)$ equipped with a conditional expectation $E :\C_0^u(\GG)\to \C_0^u(\GG)$ onto $\sX$ such that $(\id\otimes E)\circ\Delta_\GG^u=\Delta_\GG^u\circ E$.
\end{theorem}
\subsection{Normal coideals and quantum subgroups}
In the next theorem we get a 1-1 correspondence between  idempotent states of Haar type (i.e. the states corresponding to a Haar measure on a compact quantum subgroup   $\HH$ of $\GG$) and normal integrable coideals  $\sN\subset \Linf(\GG)$ preserved by the scaling group. 
\begin{theorem}\label{comchar}
Let $\sN$ be a normal integrable coideal von Neumann subalgebra of $\Linf(\GG)$ which is $\tau_t$-invariant. Then 
there exists a unique compact quantum subgroup $\HH\subset \GG$ such that $\sN=\Linf(\GG/\HH)$.
\end{theorem}
\begin{proof}
Let $\omega\in\C_0^u(\GG)^*$ be the idempotent state corresponding to $\sN$ as described in Theorem \ref{thminterg}. We define 
\[\sX^r=\bigl\{x\in\C_0(\GG)\st \omega\staru x=x\bigr\}^{\text{\tiny{norm-cls}}}\] and the universal lift of $\sX^r$ \[\sX^u=\bigl\{y\in\C^u_0(\GG)\st (\id\otimes\omega)(\Delta^u(y)) = y\bigr\}^{\text{\tiny{norm-cls}}}.\] Let $E:\Linf(\GG)\to\Linf(\GG)$ be the conditional expectation assigned to $\omega$. The normality of $\sN$ implies that $\sX^r$ is symmetric.
Indeed \[\ww^{\GG}(\I\otimes \sX^r) {\ww^\GG}^*\subset \M(\C_0(\hh\GG)\otimes \C_0(\GG))\] and 
\[(\id\otimes E)(\ww^{\GG}(\I\otimes \sX^r) {\ww^\GG}^*) = \ww^{\GG}(\I\otimes \sX^r) {\ww^\GG}^*\] i.e.\[\ww^{\GG}(\I\otimes \sX^r) {\ww^\GG}^*\subset \M(\C_0(\hh\GG)\otimes \sX^r).\] 
 Using Proposition \ref{sym} we conclude that $\sX^u$ is symmetric \[\wW^\GG(\I\otimes \sX^u){\wW^\GG}^*\subset \M(\C_0(\hh\GG)\otimes \sX^u)\]
so according to Theorem \ref{1to1coresp} there exists a compact quantum group $\HH$ of $\GG$ such that $\sX^u=\C_0^u(\GG/\HH)$. Since  $\sX^r=\Lambda_\GG(\sX^u)=\Lambda_\GG(\C_0^u(\GG/\HH))=\C_0(\GG/\HH)$ we get  $\sN=\Linf(\GG/\HH)$.
\end{proof}
\section{Open quantum  subgroups of $\GG$ and idempotent states on $\hh\GG$}\label{openidstates}
In this section  we establish a 1-1 correspondence between   open quantum subgroups of $\GG$ and      central idempotent states on $\hh\GG$. 
\begin{theorem}\label{openthm1}
Let $\HH$ be an open quantum subgroup of a locally compact quantum group $\GG$. Then there exists a conditional expectation $E:\Linf(\hh\GG)\to \Linf(\hh\HH)$ such that $(\id\otimes E)\circ\Delta_{\hh\GG}=\Delta_{\hh\GG}\circ E =(E\otimes\id)\circ\Delta_{\hh\GG}$. Conversely for a von Neumann subalgebra $\sN$ of $\Linf(\hh\GG)$ equipped with a conditional expectation $E:\Linf(\hh\GG)\to \Linf(\hh\GG)$ onto $\sN$ satisfying 
\begin{equation}\label{E}
(\id\otimes E)\circ\Delta_{\hh\GG}=\Delta_{\hh\GG}\circ E =(E\otimes\id)\circ\Delta_{\hh\GG}
\end{equation}
 there exists a unique  open quantum subgroup $\HH$ of $\GG$ such that $\sN=\Linf(\hh\HH)$.
\end{theorem}
\begin{proof}
Since $\HH$ is an open quantum subgroup of $\GG$, then it is closed (see \cite{KKS}) thus  $\Linf(\hh\HH)\subset \Linf(\hh\GG)$ is $\hh\tau_{t}$-invariant. Furthermore the restriction of $\psi_{\hh\GG}$ to $\Linf(\hh\HH)$ is semifinite \cite[Corollary 3.4]{int}. Using Theorem \ref{thminterg} there exists a conditional expectation  $E:\Linf(\hh\GG)\to \Linf(\hh\GG)$ onto $\Linf(\hh\HH)$ such that $(\id\otimes E)\circ\Delta_{\hh\GG}=\Delta_{\hh\GG}\circ E$. Since $\Linf(\hh\HH)$ is preserved by $\Delta_{\hh\GG}$ we can proceed as in the proof of Theorem \ref{thminterg} to  get $(E\otimes \id)\circ\Delta_{\hh\GG}=\Delta_{\hh\GG}\circ E$. 

Conversely, suppose $\sN$ is a von Neumann subalgebra of $\Linf(\hh\GG)$ equipped with conditional expectation $E$ satisfying \eqref{E}. 
It is easy to see that in this case $\sN$ is an invariant subalgebra, i.e. $\Delta_{\hh\GG}(\sN)\subseteq \sN\vtens\sN$.  
The following relations show that the restriction of $\psi_{\hh\GG}$ and $\varphi_{\hh\GG}$ to $\Linf(\hh\HH)$ are semifinite,
\[\begin{split}
\psi_{\hh\GG}(E(x))=(\psi_{\hh\GG}\otimes\id)\Delta_{\hh\GG}(E(x))&=(\psi_{\hh\GG}\otimes\id)(\id\otimes E)\Delta_{\hh\GG}(x)=E(\I)\psi_{\hh\GG}(x)=\psi_{\hh\GG}(x),\\
\varphi_{\hh\GG}(E(x))=(\id\otimes\varphi_{\hh\GG})\Delta_{\hh\GG}(E(x))&=(\id\otimes\varphi_{\hh\GG})(E\otimes \id)\Delta_{\hh\GG}(x)=E(\I)\varphi_{\hh\GG}(x)=\varphi_{\hh\GG}(x).
\end{split}\] 
In particular  $ (\sN,\Delta_{\hh\GG}|_{\sN}, \varphi_{\hh\GG}|_{\sN}, \psi_{\hh\GG}|_{\sN})$ is a locally compact quantum group which we shall denote by $\hh\HH$. Using \cite[Theorem 7.5]{KKS} we see that  $\HH$ can be identified with an  open subgroup of $\GG$.
\end{proof}
The next corollary is the infinite-dimensional version of \cite[Theorem 3.2]{FS} and also the generalization of \cite[Theorem 4.1]{FS} and \cite{BBS}.
\begin{corollary}\label{11op} Let $\GG$ be a locally compact quantum group. There is a 1-1 correspondence between open quantum subgroups of $\GG$ and central idempotent states $\omega$ on $\hh\GG$, i.e. idempotent states $\omega\in\C^u_0(\hh\GG)^*$ such that 
$\omega\staru\mu = \mu\staru\omega
$ for all $\mu\in\C^u_0(\hh\GG)^*$. 
\end{corollary}
\begin{proof}
If $\omega\in\C_0^u(\hh\GG)^*$ is an idempotent state satisfying $\omega\staru\mu = \mu\staru\omega
$ for all $\mu\in\C^u_0(\hh\GG)^*$ then the corresponding conditional expectation $E:\Linf(\hh\GG)\to \Linf(\hh\GG)$ satisfies \eqref{E}.

Conversely, let $\HH$ be an open quantum subgroup of $\GG$,  $\omega\in\C_0^u(\hh\GG)^*$ the corresponding idempotent state and $E$ the conditional expectation. Then $(\id\otimes E)\circ\Delta_{\hh\GG} = (E\otimes\id)\circ\Delta_{\hh\GG}$ implies $\omega\staru\mu = \mu\staru\omega$ for all $\mu\in\Linf(\hh\GG)_*$. Since $\Linf(\hh\GG)_*$ forms a two sided ideal in $\C_0^u(\hh\GG)^*$ we get 
\[\nu\staru\mu\staru\omega = \nu\staru\omega\staru\mu = \omega\staru\nu\staru\mu \] for all $\nu\in\C_0^u(\hh\GG)^*$ and $\mu\in\Linf(\hh\GG)_*$. Thus we conclude that 
\[(\nu\staru\omega\otimes\Lambda_{\hh\GG})\circ\Delta_{\hh\GG}^u = (\omega\staru\nu\otimes\Lambda_{\hh\GG})\circ\Delta_{\hh\GG}^u.\] Using Podle\'s condition for $\Delta_{\hh\GG}^u$ we conclude that $\nu\staru\omega = \omega\staru\nu$ for all $\nu\in \C_0^u(\hh\GG)^*$. 
\end{proof}
\begin{remark}
Let $\HH\subset\GG$  be an open quantum subgroup,  $\I_\HH$ the corresponding central  group-like   projection as explained  in the last paragraph of Section \ref{Prel}   and let  $\omega\in S(\C_0^u(\hh\GG))$ be the  corresponding idempotent state. It is observed in \cite{KKSS} that $\I_\HH = (\omega\otimes\id)(\Ww^\GG)$. Thus using Corollary \ref{11op} and the results of \cite{KKS} we see that a central  projection $P\in\Linf(\GG)$ is a group-like projection if and only if there exists a central idempotent state $\omega\in\C_0^u(\hh\GG)^*$ such that $P = (\omega\otimes\id)(\Ww^\GG)$. 
\end{remark}

\section{Normal coideals assigned to open quantum subgroups}\label{opencoid}
In this section we  characterize normal coideals $\sN\subset\Linf(\GG)$  corresponding to  open quantum subgroups. Roughly speaking $\sN = \Linf(\GG/\HH)$ for $\HH$ open in $\GG$ if and only if $\sN$ admits an atom.  
\begin{theorem}\label{minimal1}
Let $\sN\subset \Linf(\GG)$ be a normal coideal  
  admitting  a  minimal  projection $P\in\sN$ which is central $P\in Z(\sN)$.  Suppose that 
 \begin{equation}\label{assum5}\I\in\big\{(\id\otimes \omega)((\I\otimes P){\ww^\GG}^* (\I\otimes P))\st\omega\in\B(\Ltwo(\GG))_*\big\}^{\textrm{\tiny{$\sigma$-weak cls}}}.
 \end{equation}
  Then there exists an open quantum subgroup $\mathbb{H}\subset \GG$  such that $\sN = \Linf(\GG/\HH)$.  Conversely, if $\HH\subset\GG$ is an open quantum subgroup given by $\pi:\Linf(\GG)\to\Linf(\HH)$ then the central support $\I_\HH$ of $\pi$ is a minimal central projection in $\Linf(\GG/\HH)$ satisfying \eqref{assum5}.  
\end{theorem}
\begin{proof}
 Since $P$ is minimal and central,  there exists $x\in \Linf(\hh\GG)$ such that\[{\ww^\GG}(\I\otimes P){\ww^\GG}^*(\I\otimes P) = x\otimes P\] Thus 
 \begin{equation}\label{nor1}
 (\I\otimes P){\ww^\GG}^*(\I\otimes P) = (\I\otimes P){\ww^\GG}^*(\I\otimes P)(x\otimes\I).
 \end{equation} Applying $(\id\otimes \omega)$ to \eqref{nor1},  we conclude that
 \[
(\id\otimes\omega)((\I\otimes P){\ww^\GG}^*(\I\otimes P) )=(\id\otimes\omega)((\I\otimes P){\ww^\GG}^*(\I\otimes P) )x.
 \] 
Thus  \eqref{assum5} yields $x = \I$, i.e. 
\begin{equation}\label{comrelp}
(\I\otimes P){\ww^\GG}^* (\I\otimes P) ={\ww^\GG}^*(\I\otimes P).
\end{equation}  Let $\omega\in\B(\Ltwo(\GG))_*$. Slicing \eqref{comrelp}  with $(\omega\otimes\id)$ we get $PaP = aP$ for all $a\in\Linf(\GG)$. In particular, $a^*P = Pa^*P = (PaP)^* = (aP)^* =  Pa^*$, i.e. $P$ is central. 

Using minimality and centrality  of $P$ again, we see that for all $x\in\sN$    there exists $y\in \Linf(\mathbb{G})$ such that \[\Delta_\GG(x)(1\otimes P) = y\otimes P\] i.e. ${\ww^\GG}(x\otimes \I){\ww^\GG}^*(\I\otimes P) = y\otimes P$ which implies 
\[(x\otimes \I){\ww^\GG}^*(\I\otimes P)={\ww^\GG}^*(\I\otimes P)(y\otimes \I).\] This in turn implies that
\[(x\otimes \I)(\I\otimes P){\ww^\GG}^*(\I\otimes P)=(\I\otimes P){\ww^\GG}^*(\I\otimes P)(y\otimes \I).\]
Slicing   with $(\id\otimes\omega)$ and using \eqref{assum5} we get $x = y$.
Thus 
\begin{equation}\label{grl2}
\Delta_\GG(x)(1\otimes P) = x\otimes P
\end{equation} 
for all $x\in \sN$.
In particular, $P\in\sN$ is a group-like projection. Let $\mathbb{H} $  be  an  open quantum subgroup of $\GG$ assigned to $P$, i.e. $P = \I_\HH$.  Using \eqref{grl2} we see that $\sN\subset \Linf(\GG/\HH )$. Using \cite[Theorem 3.3]{KKS}  we get the converse containment $\Linf(\GG/\HH )\subset \sN$. Thus $\sN= \Linf(\GG/\HH )$.
 
 For the converse we use \cite[Proposition 3.2]{KKS} which yields minimality of $\I_\HH\in\Linf(\GG/\HH)$. In order to see that $\I_\HH$ satisfies \eqref{assum5} we note that, under the identification $\Ltwo(\HH)=\I_\HH\Ltwo(\GG)$,  the bicharacter $V\in\Linf(\hh\GG)\vtens\Linf(\HH)$ is  identified with $\ww^\GG(\I\otimes\I_\HH)$. In particular  
 \[\I\in\Linf(\hh\HH) = 
 \big\{(\id\otimes \omega)((\I\otimes P){\ww^\GG}^* (\I\otimes P))\st\omega\in\B(\Ltwo(\GG))_*\big\}^{\textrm{\tiny{$\sigma$-weak cls}}}.\] 
\end{proof}
Let us give  another characterizations of normal coideals of  $\sN\subset \Linf(\GG)$ corresponding to open  quantum subgroups. In order to formulate it we shall  denote $\Delta_\GG|_\sN = \alpha$ and $\beta:\sN\to\Linf(\hh\GG)\vtens \sN$ where $\beta(x)=\ww^\GG(\I\otimes x){\ww^\GG}^*$.
\begin{theorem}
Let $\sN\subset \Linf(\GG)$ be a normal coideal. If $\sN$ admits a normal $*$-homomorphism $\varepsilon:\sN\to\mathbb{C}$ such that 
\[\begin{split}(\id\otimes\varepsilon)\circ\beta &=\I\cdot \varepsilon\\
(\id\otimes\varepsilon)\circ\alpha&= \id_\sN
\end{split}\] then there exists an open quantum subgroup $\HH\subset\GG$ such that $\sN = \Linf(\GG/\HH)$. Conversely, if $\HH\subset\GG$ is open then  $\Linf(\GG/\HH)$ admits   a normal $*$-homomorphism $\varepsilon:\sN\to\mathbb{C}$ satisfying above conditions.
\end{theorem}
\begin{proof}
Let $P$ be the support of $\varepsilon$. Then $(\id\otimes\varepsilon)\circ\beta  = \I\cdot \varepsilon$ implies 
\[\ww^\GG(\I\otimes P){\ww^\GG}^*(\I\otimes P) = (\I\otimes P)\] which then implies that $P$ is central. Similarly  $(\id\otimes\varepsilon)\circ\alpha = \id_\sN$ yields 
\[\Delta_\GG(x)(\I\otimes P) = x\otimes P\] for all $x\in \sN$. In particular $P$ is a group-like projection corresponding to an open subgroup $\HH\subset \GG$. Proceeding as in the proof of Theorem \ref{minimal1} we get the identification $\sN =\Linf(\GG/\HH)$. For the converse we consider the  restriction of $\pi:\Linf(\GG)\to\Linf(\HH)$ to $\Linf(\GG/\HH)$ which  yields the required $\varepsilon$ via  $\pi(x) = \varepsilon(x)\I$ for all $x\in\Linf(\GG/\HH)$. 
\end{proof}

\section{Appendix}
In this section we will show  that a Woronowicz-closed quantum subgroup
$\HH$  of a locally compact quantum group $\GG$ has  the Haagerup property if $\GG$ has it. Assuming  coamenability of $\GG$ the result  was proved in \cite[Proposition 5.8]{DFSW}. 

Let   $\U\in\M(\C_0(\GG)\otimes\mathcal{K}(H))$ be a  unitary. We say that $\U$ is a unitary   representation of a quantum group $\GG$ on $H$ if $(\Delta_\GG\otimes\id)(\U)=\U_{13}\U_{23}$. A unitary representation $\U$ of $\GG$  admits  a unique unitary lift  $\Uu\in\M(\C^u_0(\GG)\otimes\mathcal{K}(H))$ such that \[
(\Delta_\GG^u\otimes\id)(\Uu) = \Uu_{13}\Uu_{23}
\]
satisfying $(\Lambda_\GG \otimes \id)(\Uu)=\U$. 
\begin{definition}
A unitary representation $\U\in\M(\C_0(\GG)\otimes\mathcal{K}(H))$ of $\GG$ on a Hilbert space $H$ is mixing if for all $\xi,\eta\in H$, $(\id\otimes\omega_{\xi,\eta})(\U)\in\C_0(\GG)$.
\end{definition}

\begin{lemma}\label{mix}
Let $\U\in\M(\C_0(\GG)\otimes\mathcal{K}(H))$ be a mixing representation of $\GG$. Then $\Uu\in\M(\C^u_0(\GG)\otimes\mathcal{K}(H))$ is mixing, i.e.
\begin{equation}\label{mixing1}
\big\{(\id\otimes\omega)\Uu:\omega\in B(H)_*\big\}\subset\C_0^u(\GG).
\end{equation}
\end{lemma}
\begin{proof}
Note that we have
$(\Delta_r^{r,u}\otimes \id)(\U)=\U_{13}\Uu_{23}$ and for every $\omega\in B(H)_*$,
\begin{equation}\label{mix3}\Delta_r^{r,u}((\id\otimes \omega)\U) =(\id\otimes\id\otimes\omega)(\U_{13}\Uu_{23})\end{equation} Since $(\id\otimes \omega)(\U)\in\C_0(\GG)$ for all $\omega\in\B(H)_*$ we can  use Podle\'s condition and get \[(\mu\otimes\id)(\Delta_r^{r,u}((\id\otimes \omega)\U))\in\C_0^u(\GG)\] for all $\mu\in B(\Ltwo(\GG))_*$. On the other hand  using \eqref{mix3} we get 
\begin{equation}\label{mix4}(\mu\otimes\id)(\Delta_r^{r,u}((\id\otimes \omega)\U))=(\id\otimes \omega\cdot a)(\Uu)\in\C^u_0(\GG)\end{equation}
where  $a=(\mu\otimes\id)(\U)$.   Since   $\bigl\{(\mu\otimes\id)\U\st \mu\in B(\Ltwo(\GG))_*\bigr\}^{\textrm{\tiny{norm-cls}}}$ forms a $\C^*$-algebra  acting non-degenerately on $H$  we conclude \eqref{mixing1} from \eqref{mix4}. 
\end{proof}
Using Lemma \ref{mix} and \cite[Remark 5.9]{DFSW} we get 
\begin{corollary}
Let $\HH$ be a Woronowicz-closed quantum subgroup of $\GG$. If $\GG$ has the Haagerup property. Then $\HH$ has the Haagerup property.
\end{corollary}

\subsection*{Acknowledgements} We thank A. Skalski for useful discussions on the subject of this paper and the anonymous referee  for  a  careful  reading  of  our  manuscript. 
PK was partially supported by the NCN (National Center of Science) grant
 2015/17/B/ST1/00085. FK was partially supported by the Ministry of Science of Iran.

\end{document}